\documentclass[10pt]{amsart}
\usepackage{amscd}
\usepackage{amsmath}
\usepackage{amssymb}
\usepackage{latexsym}
\usepackage{lscape}
\usepackage{xypic}
\usepackage{color}
\RequirePackage[colorlinks]{hyperref}
\hypersetup{citecolor=blue, linkcolor=blue}
\usepackage{mathrsfs}
\usepackage{overpic}
\usepackage{multirow, slashbox}
\usepackage[framed,numbered,autolinebreaks,useliterate]{mcode}
\usepackage[table]{xcolor}

\newtheorem{thm}{Theorem}[section]

\newtheorem{lem}[thm]{Lemma}

\newtheorem{lem-def}[thm]{Lemma-Definition}
\newtheorem{cor}[thm]{Corollary}
\theoremstyle{remark}

\newtheorem{rmk}[thm]{Remark}
\theoremstyle{definition}
\newtheorem{dfn}[thm]{Definition}
\newtheorem{conj}[thm]{Conjecture}

\numberwithin{equation}{section}

\newcommand{\quash}[1]{}  
\newcommand{\nc}{\newcommand}
\nc{\on}{\operatorname}

\newcommand{\bF}{{\mathbb F}}

\newcommand{\bR}{{\mathbb R}}

\newcommand{\bZ}{{\mathbb Z}}

\newcommand{\calF}{{\mathcal F}}

\nc{\al}{{\alpha}} \nc{\be}{{\beta}} \nc{\ga}{{\gamma}}
\nc{\ve}{{\varepsilon}} \nc{\Ga}{{\Gamma}} \nc{\la}{{\lambda}}
\nc{\La}{{\Lambda}}

\nc{\ad}{{\on{ad}}}

\nc{\aff}{{\on{aff}}}
\nc{\Aff}{{\mathbf{Aff}}}

\nc{\Bun}{{\on{Bun}}}

\nc{\der}{{\on{der}}}

\nc{\diag}{{\on{diag}}}

\nc{\Fl}{{\calF\ell}}

\nc{\Hol}{{\on{Hol}}}

\nc{\Id}{{\on{Id}}}

\nc{\Ind}{{\on{Ind}}}

\nc{\res}{{\on{res}}}

\nc{\tr}{{\on{tr}}}

\nc{\GSp}{{\on{GSp}}} \nc{\GU}{{\on{GU}}} \nc{\SL}{{\on{SL}}}
\nc{\SU}{{\on{SU}}} \nc{\SO}{{\on{SO}}}

\nc{\four}{{\calF our}}

\def\question#1{{}}
\author{An Huang, Shing-Tung Yau, and Mei-Heng Yueh}
\address{Department of Mathematics\\ Harvard University\\Cambridge, MA 02138.}
\title{Graph invariants from ideas in physics and number theory}
\begin{document}
\begin{abstract}
We study free scalar field theory on a graph, which gives rise to a modified version of discrete Green's function on a graph studied in \cite{CY}. We show that this gives rise to a graph invariant, which is closely related to the 2-dim Weisfeiler-Lehman algorithm for graph isomorphism testing. We complement this invariant by another type of graph invariants, coming from viewing graphs as quadratic forms over the integers. These quadratic forms respect a well-behaved "Wedge sum" of graphs, and appear to capture important graph properties regarding graph embeddings, namely the graph genus and the dual graphs.  
\end{abstract}
\maketitle
\section{Introduction}
The graph isomorphism problem is a long standing problem that is of both theoretical and practical importance, and much effort has been put into the research of this problem. For example, for many types of graphs, there are various known fast algorithms. Many graph invariants have been studied, and put into work. However, as stated in \cite{Fortin}, it is clear that there lacks a uniform and deeper understanding of this problem, thus many issues are at an unclear stage. 

One main motivation of this paper is to try to initiate a new perspective to the study of graphs, and in particular the graph isomorphism problem, from ideas familiar in quantum field theory, with the hope of contributing toward a better understanding. On the other hand, another main point of the paper is to introduce another type of graph invariants, coming from the theory of quadratic forms over $\bZ$, which for the purpose of graph isomorphism testing, seem to complement the physics idea in a certain sense that we will describe. Furthermore, it appears that some of these invariants capture important graph properties regarding graph embeddings, which we hope to be of independent interest, aside from the graph isomorphism problem.

More specifically, in section \ref{sec:PI}, we study one of the simplest quantum field theories defined on a graph\footnote{Here we use the notion quantum field theory in a sense similar to lattice gauge theory: we apply some of its very basic ideas, in a situation where there are only finitely many degrees of freedom.}, namely a real free scalar field theory, with a varying mass parameter. Its two-point correlation function gives us a version of the discrete Green's function. This function showed up in \cite{CY} for different purposes. In section \ref{sec:GI}, we explain that this Green's function directly gives rise to a particular easily computable graph invariant, which turns out to be not stronger than (but essentially very similar to) the 2-dim Weisfeiler-Lehman algorithm for graph isomorphism testing. Although this particular invariant is not sufficient, it indicates that simple ideas from physics may become useful for considerations on the graph isomorphism problem: e.g. there might be many variants of the physics idea, that give rise to more powerful graph invariants for isomorphism testing.

In section \ref{Q}, we introduce another type of graph invariants, by viewing graphs as quadratic forms over the integers: let $A$ be the adjacency matrix of a finite simple graph $G$, and take any polynomial $f$ with integral coefficients, $f(A)$ defines a quadratic form over $\bZ$. It is obvious that the isomorphism type of the quadratic form, is an invariant of the graph isomorphism type. Note that a very different idea of quadratic forms on graphs was explored in \cite{AMMN}. For the well-known examples of pairs of graphs constructed from strongly regular graphs, and the so-called CFI graphs \cite{CFI}, where 2-dim (and in some of the cases we consider, also higher dim) Weisfeiler-Lehman algorithm fails to distinguish graphs, we show extensive computational evidence, regarding how invariants from such quadratic forms could be used to distinguish these graphs in a simple and uniform way. The time complexity of the resulting algorithm in general is probably not polynomial, however it appears to be so for e.g. the CFI graphs,\footnote{There is related study in \cite{Derksen}.} and it should be worth to investigate more: see e.g. Remark \ref{complexity}. Locally, for each prime number $p$, one could view the quadratic form over the $p$-adic integers $\bZ_p$, and there one has the easily computable $p$-adic symbols, which are complete local invariants of the form. The combinatorial meaning of these $p$-adic symbols, in terms of graphs, and the classification of graphs according to these symbols, look interesting. A hint on this is described in Remark \ref{C}, and discussions before it. Furthermore, we initiate a study on the quadratic forms represented by the combinatorial Laplacian of the graph, and in particular show that, there is a commutative monoid structure, on the set of equivalence classes of these forms, given by a well-defined "Wedge sum" of graphs \ref{monoid}. We hope this structure to be useful, in studying the decomposition of graphs. Moreover, based on computer experiments, we propose Conjecture \ref{dual}, and remarks regarding its inverse statement, which aim to characterize precisely, what are the graph properties, that are captured by the quadratic form of the Laplacian: we conjecture these properties are the graph genus $g$, and the dual (multi-)graphs associated to embeddings of the graph into the genus $g$ surface. 

In section \ref{distance}, based on our graph invariant in section \ref{sec:GI}, we construct a distance function among graphs with the same number of vertices. In Appendix \ref{Experiments}, we demonstrate some initial numerical experiments regarding the behavior of this distance function. It appears that this distance function could be used to construct useful algorithms for the alignment problem of "almost isomorphic" graphs, which we plan to discuss in a coming paper \cite{HHYY}.

The graph invariant in section \ref{sec:GI} belongs to the framework of spectral graph theory: it is constructed using eigenspaces of the Laplacian matrix. As stated in \cite{CDST, CRS}, there is a hope to discover very useful invariants from this approach. Our idea is related to the idea of graph angles that is surveyed in \cite{CRS, BW}. Also, \cite{BGM} is of relevance to our idea, where the authors use the eigenspace to constrain the action of the automorphism group of the graph, on the coset space of the eigenspace. 

{\it Acknowledgements.} The authors thank CASTS (Center of Advanced Study in Theoretical Sciences) of National Taiwan University, where part of the work was done during their visit. They also thank Hung-Hsun Chen, Wen-Wei Lin and Paul Horn for their help on some preliminary testing, and thank Fan Chung, Noam Elkies, Alexander Grigor'yan, Jonathan Hanke, Rodrigo Iglesias, Greg Kuperberg, Gregory Minton, Hector Pasten, Jean-Pierre Serre, Arul Shankar, and Baosen Wu for useful discussions. M.-H. Yueh's research is partially supported by the "Graduate Student Study Abroad Program" of Ministry of Science and Technology, Taiwan, R.O.C. under grant number NSC-104-2917-I-009-002. 

\section{Free scalar field theory on a graph} \label{sec:PI}

Let $G$ be a graph with $|G|=n$ vertices, choose an arbitrary labeling of the vertices by $V=\{x_1,x_2, \ldots,x_n\}$, and let $M$ denote its $n\times n$ (combinatorial) Laplacian matrix under this basis: for $i\neq j$, the $i,j$-th entry is equal to $-1$ if there is an edge between $x_i$ and $x_j$, and is equal to $0$ otherwise. The diagonal entries are the degrees of the vertices, so that the sum of any column of $M$ is equal to $0$. From the definition, the matrix $M$ is symmetric. $M$ represents the combinatorial Laplacian operator under the dual basis.  Let $\lambda_k$, $k=0,1, \ldots, m$ denote the set of different eigenvalues of $M$ by increasing order. For each $k$, let the column vectors $\phi_k^1,\ldots,\phi_k^{l_k}$ denote an orthonormal basis of the corresponding eigenspace $E_k$.

Consider an Euclidean real scalar field theory on the graph $G$: the space of fields is then the space of all real valued functions on vertices of $G$, which is an $n$-dimensional real vector space. We write the free field Lagrangian with a mass parameter $u={\rm mass}^2$ in direct analogy with the familiar Lagrangian in the continuous situation:
\begin{equation}
\frak L=\sum_{e\in E}(\nabla_e\phi)^2+u\phi^2, 
\end{equation}
where $\nabla_e$ is the graph gradient with respect to a directed edge $e$, $(\nabla_e\phi)^2$ is independent of the choice of the orientation of $e$, and $E$ is the set of edges of $G$. One can consult \cite{CGY} for these notations. We have the usual Green's formula
\begin{equation}
\int_G (\nabla_e\phi)^2 dx=\int_G \phi \Delta \phi dx, 
\end{equation}
where $\Delta$ is the Laplacian.

As the same with usual quantum field theory (QFT) on a manifold, we consider two-point correlation functions defined by
\begin{equation}\label{2cor}
\left<\phi(x)\phi(y)\right>=\frac{\int \phi(x)\phi(y)e^{-\int \phi (\Delta+u) \phi dx}D\phi}{\int e^{-\int \phi (\Delta+u) \phi dx}D\phi}.
\end{equation}
We allow $x$ and $y$ to be equal, as there will be no short distance problems in our situation. This is a finite dimensional path integral of the type that is often used as toy model to introduce the Feynman rules in physics textbooks, and it is free of divergences. However, in our simple situation here, this is our path integral. We know very well how to evaluate this by undergraduate calculus with familiar result: the denominator equals the determinant of the Laplacian to the power $-\frac{1}{2}$, which cancels with a factor coming from the numerator. Up to a nonzero constant scalar, what is left, is a sum over different eigenvalues of the form
\begin{equation}\label{2cor_fun}
\mathcal{T}^G_u(x,y) \equiv \sum_{k=0}^m \frac{t_k(x,y)}{\lambda_k+u}, 
\end{equation}
which may be viewed as a discrete version of the Fourier transform of the D'Alembert propagator, the familiar result in usual QFT. The individual $t_k(x,y)$ for each eigenvalue may be recovered as residues near different poles of the two-point correlation function, as we vary the parameter $u$.

It is straightforward to check that the function $\left<\phi(x)\phi(y)\right>$ satisfies a discrete version of the quantum equation of motion
\begin{equation} \label{QEoM}
L_x\left<\phi(x)\phi(y)\right>=\delta_{x,y},
\end{equation}
where $L_x$ is the Laplacian operator on coordinate $x$, and the delta function $\delta_{x,y}$ on a graph is given by
\begin{eqnarray*}
\delta_{x,y}=\begin{cases}
0, \ &x\neq y, \\
1, \ &x=y.
\end{cases}
\end{eqnarray*}
Therefore, we call the two-point correlation function as a discrete Green's function. In addition, upon a choice of labeling of vertices, as we have done, Equation \eqref{QEoM} becomes the statement that $\left<\phi(x)\phi(y)\right>$ as a matrix, is the inverse of $M+uI$. So obviously, it determines the graph up to isomorphism.

\begin{rmk}
The two-point correlation function determines the graph up to isomorphism, thus it also determines the QFT on the graph, and therefore all of its correlation functions. This can be viewed as a baby version of Wick's theorem in the graph case.
\end{rmk}

Furthermore, one can then study various operations on graphs, and try to see how the two-point correlation function changes accordingly. This is interesting because, theoretically, it is almost always important to understand how invariants change under important operations. On the other hand, as will be discussed in section \ref{distance}, the two-point correlation function can provide a measure on when two given graphs are considered "almost isomorphic", which may be useful in practice. For example, if we have a large data presented as a big graph, one should expect that the data given may contain a little marginal error, and so being able to make sense of and detect "almost isomorphic" graphs looks to be a practically important problem. 

For example, suppose we delete an edge (adding an edge will be just the opposite, of course) between two vertices $x_1$ and $x_2$, and get a new graph, called $G_2$. Let us try to write down the two-point correlation function for $G_2$ in terms of data of $G$ and the two vertices $x_1$ and $x_2$. From the form of \eqref{2cor}, we know that this operation may only possibly affect the term $e^{\int \phi \Delta \phi dx}$. For this term, at any vertex other than $x_1$ and $x_2$, the action of the Laplacian is unaffected by definition. At $x_1$, the integral $\int \phi \Delta \phi dx$ changes by $\phi(x_1)(\phi(x_1)-\phi(x_2))$, and at $x_2$, the integral changes by $\phi(x_2)(\phi(x_2)-\phi(x_1))$. Therefore, the two-point correlation function for $G_2$ can be expressed as
\begin{equation}\label{2}
\left<\phi(x)\phi(y)\right>_{G_2}=\frac{\int \phi(x)\phi(y)e^{\int \phi \Delta \phi dx}e^{(\phi(x_1)-\phi(x_2))^2}D\phi}{\int e^{\int \phi \Delta \phi dx}e^{(\phi(x_1)-\phi(x_2))^2}D\phi}.
\end{equation}
Again, the above can be explicitly calculated by Gaussian integrals, and one may then compare it with the two-point correlation function of $G$, and analyze the difference in various situations. One elementary observation is that, roughly speaking, difference of values of eigenfunctions at vertices $x_1$ and $x_2$ contribute to the difference of two-point correlation functions. Furthermore, the two-point correlation function is more sensitive to the difference at smaller eigenvalues. This is consistent with the physics picture: smaller eigenvalues correspond to lower energy modes, and if the low energy modes for two graphs are close, then we have a sense that these two graphs are close to each other.

\begin{rmk}
The individual functions $t_k(x,y)$ will change in a more complicated manner, and probably one should not expect a particularly nice formula for the change of $t_k(x,y)$ similar to \eqref{2}, because e.g. even the number of distinct eigenvalues and the dimension of eigenspaces may jump, and there may be complications from cross terms. The combination $\left<\phi(x)\phi(y)\right>$ takes into account all of these and the change of it can be presented by the simple formula above.

It looks quite possible that one may study more elaborated quantum field theories on a general graph, especially with the help of various topological and geometrical concepts for graphs that are developed for graphs recently \cite{K0, GLMY, GLMY2, HY}.
\end{rmk}

\section{A graph invariant} \label{sec:GI}

Suppose we have another graph $G_1$ with $n$ vertices, and upon a choice of an arbitrary labeling of the vertices, we get another Laplacian matrix $M_1$. The problem of whether $G$ and $G_1$ are isomorphic graphs, amounts to the linear algebra question of whether there exists a permutation matrix $P$, such that $P^\top MP=M_1$. (Note that $P^\top=P^{-1}$.) In spectral graph theory, people study the real spectrum of $M$, as an invariant of the graph under isomorphisms, however, the spectrum itself is not sufficient for the graph isomorphism problem. Two graphs can have the same real spectrum but fail to be isomorphic, and these are called cospectral graphs. On the other hand, the eigenfunctions contain much more information than just the eigenvalues. The apparent question of dealing with the eigenfunctions or eigenspaces is that they are not preserved under graph isomorphisms, but instead, the eigenspaces also transform by permutations. So, in order to use them appropriately in the graph isomorphism problem, one needs to find suitable invariants associated with the eigenfunctions.

We denote $t_k(x,y)=\sum_{i=1}^{l_k}\phi_k^i(x)\phi_k^i(y)$, and $T(x,y)=\left<t_0(x,y), \ldots, t_m(x,y)\right>$. It is obvious that the vector function $T(x,y)$ does not depend on the choice of the orthonormal basis, and it can be constructed directly from the graph Laplacian independent of the choice of a labeling of vertices, therefore it is an intrinsically defined function on $G\times G$. The set of $1\times (m+1)$ vectors $T(x,y)$ counting multiplicity, marked by each corresponding eigenvalue, where $x,y$ range among all pairs of vertices of $G$, is therefore an invariant of the graph, which we denote by $ST$. This invariant is clearly polynomial time computable, and furthermore the elements of this set can be ordered in order for comparisons. In the following, we explain how this invariant arises directly from the free scalar field theory, and how it is related to the 2-dimensional Weisfeiler-Lehman algorithm.

\begin{rmk}
The above method is linear algebra that can also work for suitable variations of the Laplacian matrix, for example, the normalized Laplacian. Furthermore, the discussion can actually be applied to more general situations, such as multi-graphs.
\end{rmk}

As the two-point function matrix is the inverse of $M+uI$, by the adjugate matrix formula of an inverse matrix, we have
\begin{equation}
\left<\phi(x)\phi(y)\right>=\frac{(-1)^{x+y}A_{y,x}}{\det(M+uI)}, 
\end{equation}
where $A_{y,x}$ is the $y,x$-th cofactor of $M+uI$, which is a polynomial in $u$ of integral coefficients of degree less than $n$. Since our discrete Green's function can be written as an integral of the heat kernel which is positive, one expects $\left<\phi(x)\phi(y)\right>$ to be positive. In fact, one has the following stronger fact. 
\begin{lem}\label{positivity}
All coefficients of the polynomial $(-1)^{x+y}A_{y,x}$ are positive.
\end{lem}
\begin{proof}
This is a simple verification by induction.
\end{proof}

We consider the graph invariant given by the set of values (actually a set of functions of $u$) of the two-point correlation function, counting multiplicities. We have
\begin{equation}\label{f}
\left<\phi(x)\phi(y)\right>=\sum_{k=0}^m\sum_{i=1}^{l_k}\frac{\phi_k^i(x)\phi_k^i(y)}{\lambda_k+u}, 
\end{equation}
and by basic linear algebra, more generally,
\begin{equation}\label{g}
(M+uI)^{\alpha}=\sum_{k=0}^m\sum_{i=1}^{l_k}(\phi_k^i(x)\phi_k^i(y))(\lambda_k+u)^{\alpha}, 
\end{equation}
for any $\alpha\in \bR$. Note that one can take such arbitrary powers of a positive semi-definite matrix.

Therefore, if for two graphs $G$ and $G_1$, the invariant we are considering are the same, it will mean that there exists a permutation $Q$ of $n^2$ elements acting linearly on $n\times n$ matrices by permuting the corresponding elements, such that
\begin{equation}
Q\left<\phi(x)\phi(y)\right>_G=\left<\phi(x)\phi(y)\right>_{G_1}.
\end{equation}
By the above equation combined with taking residues of \eqref{f}, we have, for every $k$,
\begin{equation}
Q\left<\sum_{i=1}^{l_k}\phi_k^i(x)\phi_k^i(y)\right>_G=\left<\sum_{i=1}^{l_k}\phi_k^i(x)\phi_k^i(y)\right>_{G_1}.
\end{equation}
Therefore by \eqref{g}, we have
\begin{equation} \label{r}
Q(M+uI)^{\alpha}=(M_1+uI)^{\alpha},
\end{equation}
for all $\alpha\in \bR$.

\begin{rmk}
Conversely, one convinces oneself easily that, if \eqref{r} holds, then the pair of graphs are cospectral, and our graph invariant takes the same value for the pair.
\end{rmk}

\eqref{r} gives interesting identities. e.g. Taking $\alpha=0$, one derives that $Q$ preserves the diagonal. Taking $\alpha$ to be positive integers, and $u=0$, one gets infinitely many identities with more or less clear combinatorial meaning.

On the other hand, as it is clear from the above derivation, the set of $1\times (m+1)$ vectors $T(x,y)$ marked by eigenvalues, which we denoted by $ST$, as an invariant of the graph, is equivalent to the above set of values of two-point correlation functions.

\begin{rmk}
It may be expensive to compute the cofactors $A_{y,x}$ as a polynomial in $u$. However, one can instead take $n$ different values of $u$, and compute the corresponding $n$ values of $A_{y,x}$, which determine $A_{y,x}$ uniquely as it is a polynomial of degree less than $n$. This can be done quickly, and will be used in section \ref{distance} for constructing practically computable distance functions.
\end{rmk}

It turns out that, \eqref{r} is a consequence of the 2-dimensional Weisfeiler-Lehman algorithm, as shown in Theorem 3 of \cite{AIP}. It is not yet clear to us if 2-dimensional Weisfeiler-Lehman is strictly stronger than \eqref{r}, and if so, to what extent. As a consequence, the strongly regular graphs of a given type, and the famous pairs of graphs constructed in \cite{CFI} (which we will refer to as CFI graphs in the following) have the same $ST$ invariants. On the other hand, it is conceivable that a variation of the physics construction (e.g. considering the set of values of $2k$ point functions) may be closely related with higher dimensional Weisfeiler-Lehman algorithm, and its variations. 

\section{Quadratic form invariants for graphs}\label{Q}

Let $A$ be the adjacency matrix of a graph $G$, and $f\in\bZ$[$x$] a polynomial with integral coefficients. As a permutation matrix lies in $GL(n,\bZ)$, the isomorphism class of the quadratic form over $\bZ$, represented by $f(A)$, is an invariant of the graph.

Suppose $G$ and $G_1$ are isomorphic graphs, then there exists a permutation matrix $P$, such that $P^TAP=A_1$, where $A_1$ is the adjacency matrix of $G_1$. As $P\in GL(n,\bZ)\cap O(n,\bR)$, it is automatic that $P^Tf(A)P=f(A_1)$ for all polynomials $f\in\bZ$[$x$]. On the other hand, we have the following very easy observation:

\begin{lem}\label{Obs}
Suppose there exists $P\in GL(n,\bZ)\cap O(n,\bR)$, such that $P^TAP=A_1$, then $G$ is isomorphic to $G_1$.
\end{lem}

\begin{proof}
As $P\in GL(n,\bZ)\cap O(n,\bR)$, $P$ consists of orthonormal rows, where each row has exactly $1$ nonzero element, which is either $1$ or $-1$. Therefore, there exists a diagonal matrix $D$, whose diagonal entries are either $1$ or $-1$, such that $DP$ is a permutation of $n$ elements. So, $D^TAD$ is obtained from $A_1$ via a permutation. As all nonzero entries of $A$ and $A_1$ are positive, this forces $D^TAD=A$.
\end{proof}

Here is an heuristic idea why such invariants might be useful for graph isomorphism problem: suppose $G$ and $G_1$ are not isomorphic, and suppose for any $f\in\bZ$[$x$], $f(A)$ and $f(A_1)$ represent the same quadratic form over the integers, then there must exist a $T_f\in GL(n,\bZ)$ to transform one form to another. On the other hand, in view of lemma \ref{Obs}, $T_f$ cannot be an orthogonal matrix, so $f(A)$ and $f(A_1)$ being isomorphic for one $f$, does not seem to directly imply the same statement for any different $f$, which is rather hard to imagine. So intuitively, if $G$ and $G_1$ are not isomorphic, it looks reasonable hope that at least for some $f$, $f(A)$ and $f(A_1)$ will represent different quadratic forms over $\bZ$.

On the other hand, for pairs of graphs that the ST invariant cannot distinguish, they tend to be graphs which are very regular: e.g. the strongly regular graphs, and the CFI graphs. As both types of graphs are highly constrained, it might be reasonable to hope that, some very simple $f$ might provide quadratic forms, which are sufficient to tell apart all these graphs. We provide some extensive computational evidence in the following, that suggests this hope might be true in a strong sense: e.g., for any generic integer $k$, $f=(x+k)^2$ seems to do the job.\footnote{Actually, we also observe in computer experiments that, the canonically defined combinatorial Laplacian seems to do the job as well, and that also looks much quicker to compute.}

To compare two nondegenerate quadratic forms over $\bZ$ with the same discriminant, one can first compare the forms over the $p$-adic integers $\bZ_p$ for each prime $p$ dividing the discriminant, where the so-called $p$-adic symbols are complete local invariants, and are very easy to compute. If the two forms are equivalent over $\bZ_p$ for each such $p$, then they are said to be in the same genus. If the forms are positive definite, and have large dimension, and large discriminant, a genus often contains a huge number of forms, as predicted by the Smith-Minkowski-Siegel mass formula \cite{Siegel}. On the other hand, if the forms are indefinite, and the dimension is at least 3, then there is a so-called "spinor genus" that refines the genus, and is easily computable, and is a complete invariant that determines the form over $\bZ$ \cite{CS}.

At the time we write this article, we do not know of any existing software, that conveniently compares the spinor genus of indefinite forms, therefore we use positive (semi-)definite $f(A)$ in computer experiments described in the following.

We use Magma to check the quadratic forms represented by $(A+mI)^2$ for strongly regular graphs\footnote{The data of adjacency matrices of strongly regular graphs can be downloaded at \url{http://www.maths.gla.ac.uk/~es/srgraphs.php}}, where $m$ is chosen to be $0$ or $2$, and in order to test degenerate cases, we also choose $m$ to be negative of the eigenvalues of $A$. The source data of strongly regular graphs are classified according to the number of vertices $n$, the degree $k$, the number of common neighbors for each pair of adjacent vertices $\lambda$, and the number of common neighbors for each pair of non-adjacent vertices $\mu$. Such set of strongly regular graphs is denoted by ${\rm srg}(n,k,\lambda,\mu)$. The number of distinct quadratic forms $(A+mI)^2$ with respect to ${\rm srg}(n,k,\lambda,\mu)$ is written in Table \ref{tab:srg}. The number of distinct graphs for each equivalence class of quadratic forms is also recorded in the square brackets in Table \ref{tab:srg}. The Magma code for checking isomorphism of two quadratic forms is written in Appendix \ref{code:magma}.  

 It is well known that graphs in ${\rm srg}(n,k,\lambda,\mu)$ all share the same spectrum, and the adjacency matrix of any graph in ${\rm srg}(n,k,\lambda,\mu)$ satisfies the identity
\[
A^2 + (\mu-\lambda) A + (\mu-k)I = \mu J,
\]
where $J$ is the all-ones matrix of dimension $n$. 

In particular, when $\lambda=\mu$, $A^2$ will be the same for any graph in ${\rm srg}(n,k,\lambda,\mu)$. So in this case, we test the form $(A+2I)^2$. When $\lambda \neq \mu$, we test the form $A^2$.

According to Table \ref{tab:srg}, the strongly regular graphs for which $\lambda \neq \mu$ can all be distinguished by the quadratic form $A^2$. The strongly regular graphs for which $\lambda = \mu$ can all be distinguished by the quadratic form $(A+2I)^2$. 

For some degenerate cases, meaning we choose $m$ to be the negative of an eigenvalue of $A$, the quadratic form $(A+mI)^2$ fails to distinguish every strongly regular graph of a given type ${\rm srg}(n,k,\lambda,\mu)$, especially for which the multiplicity of eigenvalue $-m$ is large.

Furthermore, we compute the $p$-adic symbols\footnote{The format of the $p$-adic symbols is according to \cite{Conway}.} of the quadratic form $A+I$, for strongly regular graphs. The result is written in Table \ref{tab:p-adic}. In addition, we use Sage to compute the $p$-adic symbols of the quadratic form $A$ for strongly regular graphs. The result is written in Table \ref{tab:p-adic_srg} and Table \ref{tab:p-adic_srg2}. The Sage code for computing $p$-adic symbols of quadratic forms is written in Appendix \ref{code:sage_padic}. One sees that the discriminant group ($p$-adic symbols) tends to have only a few possibilities, for each ${\rm srg}(n,k,\lambda,\mu)$.

 To construct CFI regular graphs\footnote{We actually tested also non-regular CFI graphs, which behave in a similar way.}, first we use GENREG\footnote{The software can be downloaded at \url{http://www.mathe2.uni-bayreuth.de/markus/reggraphs.html}.} \cite{M1999} to generate the set of all $k$-regular graphs of $n$ vertices ${\rm reg}(n,k)$. For each graph in ${\rm reg}(n,k)$, we construct the corresponding CFI pair. For convenience, we denote the set of all the CFI pairs with respect to ${\rm reg}(n,k)$ by ${\rm CFI}({\rm reg}(n,k))$. 

Similarly, we use Sage to compute the $p$-adic symbols of the quadratic form $A$ for CFI graphs of vertex number less than or equal to $100$, for which $\det A\neq 0$. The result is written in Table \ref{tab:p-adic_CFI} and Table \ref{tab:p-adic_CFI2}. According to the result, all the tested CFI pairs can be distinguished by the $2$-adic symbol of the quadratic form $A$. There are some recent related studies on this in the literature, see e.g. \cite{Derksen} and \cite{DB}. On the other hand, for odd primes $p$, the $p$-adic symbols of each CFI pair are identical.  

Lastly, Figure \ref{fig:CFI} shows the relationship between the number of vertices of ${\rm CFI}({\rm reg}(n,k))$, for $n=4, \ldots, 12$, and the number of distinct prime factors of $\det(A+2I)$. From Figure \ref{fig:CFI}, we observe that the number of distinct prime factors is very small. In fact, the largest prime factor is also very small compared to the determinant, which we do not display in the figure.  

We can try to read off the combinatorial meaning of some information contained in these $p$-adic symbols, directly from definitions. e.g., apparently, there is a bijection from the kernel of $A$ mod 2, as a vector space over the finite field $\bF_2$, to the set of subsets of the vertex set of the graph, such that every vertex of the graph is connected with an even number of vertices in the subset, whereas the dimension of the kernel of $A$ mod 2, is the simplest piece of information contained in the 2-adic symbol of $A$. We hope a better understanding of the combinatorial meaning of these $p$-adic symbols, will help us understand the computer experimental results for them.\footnote{We thank Noam Elkies and Jean-Pierre Serre for some very preliminary discussions on these issues.}

\begin{rmk}\label{C}

One can see from Table \ref{tab:srg} and Table \ref{tab:p-adic_srg} that, the 4-element set ${\rm srg}(28, 12, 6, 4)$ is split into a subset of 3 graphs, and a subset of a single graph, by either the degenerate quadratic form $(A+2I)^2$ over $\bZ$, or the 2-adic symbol for $A$. One can check that, the one graph that is singled out in either way, is exactly the line graph of the complete graph $K_8$. So the other 3-element subset is the set of the so-called Chang graphs \cite{Chang}. 

\end{rmk}

From the experiments, we see that for complete sets of strongly regular graphs parametrized by four parameters $n,k,\lambda,\mu$, and for all $p,q\in\bZ$ that we have tested, the forms $A+pI+qJ$ always seem to have only a few different types of discriminant groups. On the other hand, for generic $p,q$, the isomorphism class of the form $A+pI+qJ$ over $\bZ$ seems to be always able to distinguish every graph in the set. We do not include all these detailed test results in this article, due to space.

Locally, these test results imply that the information of the $p$-adic symbols, in these cases of strongly regular graphs, is very strongly constrained by the 4 parameters.

From a number theory point of view, the properties of these quadratic forms look curious: in particular, the authors do not know any obvious reason, why these forms tend to be locally equivalent, for all $p,q$.

\begin{rmk}\label{complexity}
The indefinite forms $A$ and $A_1$ can not distinguish strongly regular graphs in general, as one readily checks that their determinants are in general not big enough for the spinor genus to contain more than one class in its genus: see e.g. Corollary 22 on page 395 of \cite{CS}. On the other hand, computer experiments seem to suggest that, comparing the definite forms for strongly regular graphs as we did, is in general much faster than comparing the same forms coming from pairs of random graphs of the same size.
\end{rmk}

Next, we consider the quadratic form represented by the combinatorial Laplacian matrix $M$, which is again, an obvious graph invariant. Let $G_1,G_2$ be two graphs, and $X$ be a vertex of $G_1$, and $Y_1,...,Y_k$ be a set of vertices of $G_2$, we define a new graph $G$ to be the disjoint union of $G_1$ and $G_2$, together with an edge between each pair of vertices $X,Y_i$, $i=1,...,k$. We have the following

\begin{lem}\label{move}
The isomorphism class of the quadratic form represented by the combinatorial Laplacian of $G$, is independent of the choice of $X$.
\end{lem}

\begin{proof}
Let $n_1$ and $n_2$ denote the vertex numbers of $G_1$ and $G_2$, respectively. After a possible vertex permutation, the Laplacian matrix $M$ of $G$ is formed by two diagonal blocks $M_1+kE_{n_1,n_1}^{n_1,n_1}$ and $M_2+\sum_{i=1}^kE_{i,i}^{n_2,n_2}$, together with additional $-1$ at positions $(n_1,n_1+1),...(n_1,n_1+k)$ and $(n_1+1,n_1),...,(n_1+k,n_1)$, where $E_{i,j}^{s,t}$ denote the $s\times t$ matrix with entry $1$ at position $(i,j)$, and entry $0$ everywhere else. $M_1, M_2$ denote the Laplacian matrices of $G_1$ and $G_2$ respectively. We do the operation of adding the last $n_2$ rows to row $n_1$, and then adding the last $n_2$ columns to column $n_1$, the resulting matrix $N$ is formed by the block $M_1$, and the block $M_2+\sum_{i=1}^kE_{i,i}^{n_2,n_2}$. 

We have that $M\cong N$ as quadratic forms, and the isomorphism class of $N$ is obviously independent of the choice of $V_1$. Therefore, the isomorphism class of $M$ is independent of the choice of $V_1$. 
\end{proof}

For ease of terminology, let us call the quadratic form represented by the combinatorial Laplacian of a graph $G$, as the quadratic form of $G$. Next, we define the wedge sum graph $G_1\vee_{X,Y}G_2$ of $G_1$ and $G_2$ w.r.t. the vertex $X$ of $G_1$, and the vertex $Y$ of $G_2$, to be the graph formed by the quotient of the disjoint union of $G_1$ and $G_2$, by identifying $X$ and $Y$. We have the following

\begin{cor}\label{ind}
The isomorphism class of the quadratic form of $G_1\vee_{X,Y}G_2$, in independent of the choices of $X$ and $Y$.
\end{cor}

\begin{proof}
Let us denote the set of neighbors of $Y$ in $G_2$, by $\{Y_1,...,Y_k\}$. Observe that $G_1\vee_{X,Y}G_2$, is the same graph as that described before lemma \ref{move}, w.r.t. $G_1$, the subgraph of $G_2$ formed by deleting the vertex $Y$ and all edges connecting with it, $X$ and $Y_1,...,Y_k$. Therefore, by lemma \ref{move}, the isomorphism class of the quadratic form in question is independent of the choice of $X$. By symmetry, it is also independent of $Y$.
\end{proof}

\begin{dfn}
We impose an equivalence relation among graphs, by declaring two graphs to be equivalent, if and only if the quadratic forms represented by their Laplacian matrices are equivalent as quadratic forms, and call the set of such equivalence classes $G_f$. 
\end{dfn}

\begin{thm}\label{monoid}
The above wedge sum of graphs is well defined on the equivalence classes, and it gives rise to a structure of commutative monoid on the equivalence classes.
\end{thm}

\begin{rmk}
As a result, any two trees of the same vertex number are in the same equivalence class, as they can be formed via step-by-step wedge sums of two-vertex trees. The reader is invited to check this by hand.
\end{rmk}

\begin{proof}
First, the equivalence class of $G_1\vee_{X,Y}G_2$ is independent of the choices of $X,Y$ by corollary \ref{ind}. Denote as before, $n_1=|G_1|$, and $n_2=|G_2|$.

Next, we view $G_1\vee_{X,Y}G_2$ as in the proof of corollary \ref{ind}. Then by doing the row and column operations as in the proof of lemma \ref{move}, the quadratic form of $G_1\vee_{X,Y}G_2$ can be represented by a block diagonal matrix, where one block is the combinatorial Laplacian $M_1$ of $G_1$, and the other block is of dimension $n_2-1$, and is determined by $G_2$ and $Y$. Next, suppose $G_1$ and $H_1$ are two graphs in the same equivalence class, and $X'$ be any vertex of $H_1$. The quadratic form of $H_1\vee_{X',Y} G_2$ can also be represented in the same way, where one block is the combinatorial Laplacian $M_1'$ of $H_1$, and the other block is determined by $G_2$ and $Y$, identical to that of $G_1\vee_{X,Y}G_2$. Since we have $M_1\cong M_1'$ as quadratic forms, there exists $T\in GL(n_1,\bZ), |\det(T)|=1$, such that $T^tM_1T=M_1'$, therefore the forms of $G_1\vee_{X,Y}G_2$ and $H_1\vee_{X',Y}G_2$ are equivalent, via the block diagonal matrix $\diag(T,I_{n_2-1})$. This proves our wedge sum of graphs is well defined on the equivalence classes.

Next, corollary \ref{ind} implies that on the equivalence classes, the sum is associative and commutative. Furthermore, the one-vertex graph $P$ affords the identity element, and $G_f$ becomes a commutative monoid.
\end{proof}

Let $G$ be a connected graph other than the single-vertex graph $P$, we call $G$ to be a simple graph, iff $G$ cannot be written as a sum $G_1\vee_{X,Y} G_2$ of two strictly smaller graphs. Obviously the simple graphs, together with $P$, generate the submonoid of $G_f$ consisting of connected graphs. We have the following

\begin{lem}\label{decomposition}
Any connected graph other than $P$ has a unique decomposition into a wedge sum of simple graphs, which we call its simple components.
\end{lem}

\begin{proof}
First, an easy induction shows such a decomposition always exists.

Next, to prove uniqueness, we consider the set of gluing vertices in the graph $G$: namely, the set of vertices where the graph can be separated as a wedge sum of two components, glued along the vertex. If there is no such vertex, then by definition the graph is simple, and there is nothing to prove. Otherwise, pick any such vertex $V$, then $G$ is separated into components, glued along $V$, which we call components of $G$, w.r.t. $V$. Obviously, any connected subgraph of the connected graph $G$ containing $V$ cannot be a simple graph, so the simple graphs in the decomposition of $G$, which are obviously connected, all appear as subgraphs of components of $G$, w.r.t. $V$. Suppose there is another graph $G'$ isomorphic to $G$ as a graph, which is also written as a wedge sum of simple graphs. Then under this graph isomorphism, $V$ must be mapped to a gluing vertex $V'$ of $G'$. Next, the components of $G$ w.r.t. $V$ must map bijectively to isomorphic components of $G'$ w.r.t. $V'$.  Again, any simple graph in the decomposition of $G'$, must appear as a subgraph of components of $G'$ w.r.t. $V'$. By induction, each such component has a unique decomposition into simple graphs, thus the decomposition of $G'$ must be the same as that of $G$.
\end{proof}

\begin{rmk}
Therefore, to understand the quadratic form of a connected graph, it suffices to understand the forms of its simple components, and we do not need to worry about how the simple components glue. 
\end{rmk}

To go further, from computer experiments that will be described below, we propose the following

\begin{conj}\label{planar_dual}
Let $G_1$ and $G_2$ be two connected planar graphs, and suppose there exist embeddings of the two graphs into the 2-sphere, such that the dual graphs w.r.t. the embeddings are isomorphic as multi-graphs, then the quadratic forms over $\bZ$, represented by the combinatorial Laplacian of $G_1$ and $G_2$ are equivalent.
\end{conj}

Furthermore, one could try to remove the planar graph restriction: in general, any graph can be embedding into a higher genus surface, and the minimal genus of such surfaces is called the genus of the graph. Given any such embedding, the dual (multi-)graph can obviously be defined in the same way, as that in the planar case. The above conjecture could be extended to the following

\begin{conj}\label{dual}
Let $G_1$ and $G_2$ be two connected graphs of the same genus $g$, and further suppose that $G_1$ and $G_2$ have embeddings into the genus $g$ surface $\Sigma_g$, such that the corresponding dual graphs are isomorphic as multi-graphs, then the quadratic forms over $\bZ$, represented by the combinatorial Laplacian of $G_1$ and $G_2$ are equivalent..
\end{conj}

\begin{rmk}
Note that, if one drops the condition that $G_1$ and $G_2$ having the same genus, the conjecture would become false: e.g. take the 3-path and the 3-cycle, both can be embedded into a torus, such that the dual graph consists of a single vertex, and 3 self-edges. However the quadratic forms of the 3-path and the 3-cycle are not equivalent.
\end{rmk}

\begin{rmk}
In addition, from computer experiments, one could wonder to what extent, the inverse statement could be true. We have some computational evidence supporting that the inverse statement (or perhaps some variation of that) might be true.
\end{rmk}

\begin{rmk}
The above conjectures link the geometry and combinatorics of graphs, with the arithmetic of quadratic forms. In particular, note that for connected graphs, Corollary \ref{ind} is essentially a consequence of the conjectures. 
\end{rmk}

Now we exhibit a few examples whose dual graphs with respect to certain embeddings are isomorphic, for which the conjectures are tested to be true.
Figure \ref{fig:g=0} shows a pair of planar graphs (in color blue) constructed by "gluing" two $3$-cycles with a $4$-cycle, embedded into a sphere.
Figure \ref{fig:g=1} shows embeddings into a torus of a pair of genus $1$ graphs constructed by gluing two $4$-cycles with a complete graph on $5$ vertices. 

\begin{figure}[h!]
\includegraphics[height=4cm]{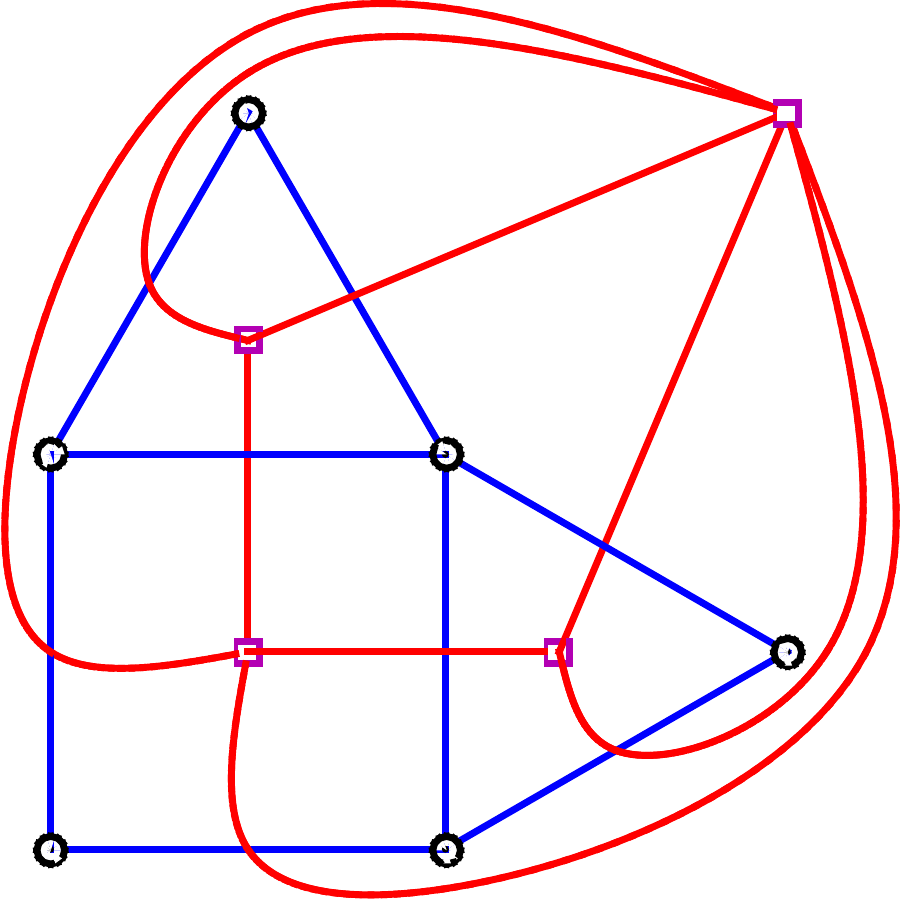}
\includegraphics[height=4cm]{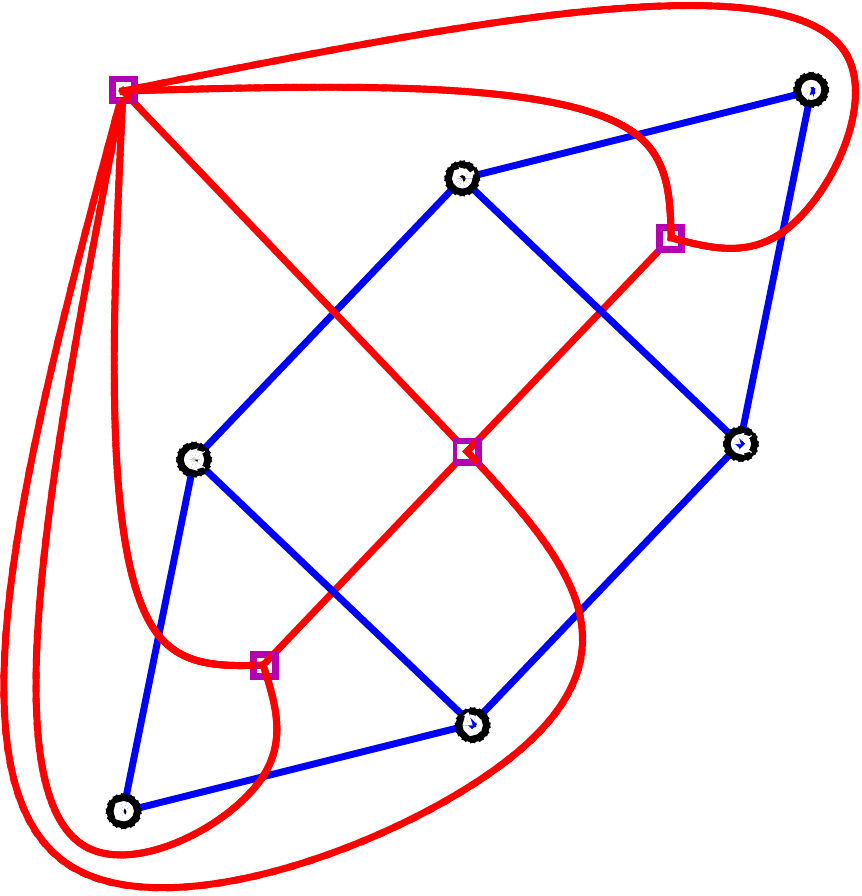}
\caption{A pair of planar graphs embedded into the sphere, with isomorphic dual graphs}
\label{fig:g=0}
\end{figure}

\begin{figure}[h!]
\includegraphics[height=4cm]{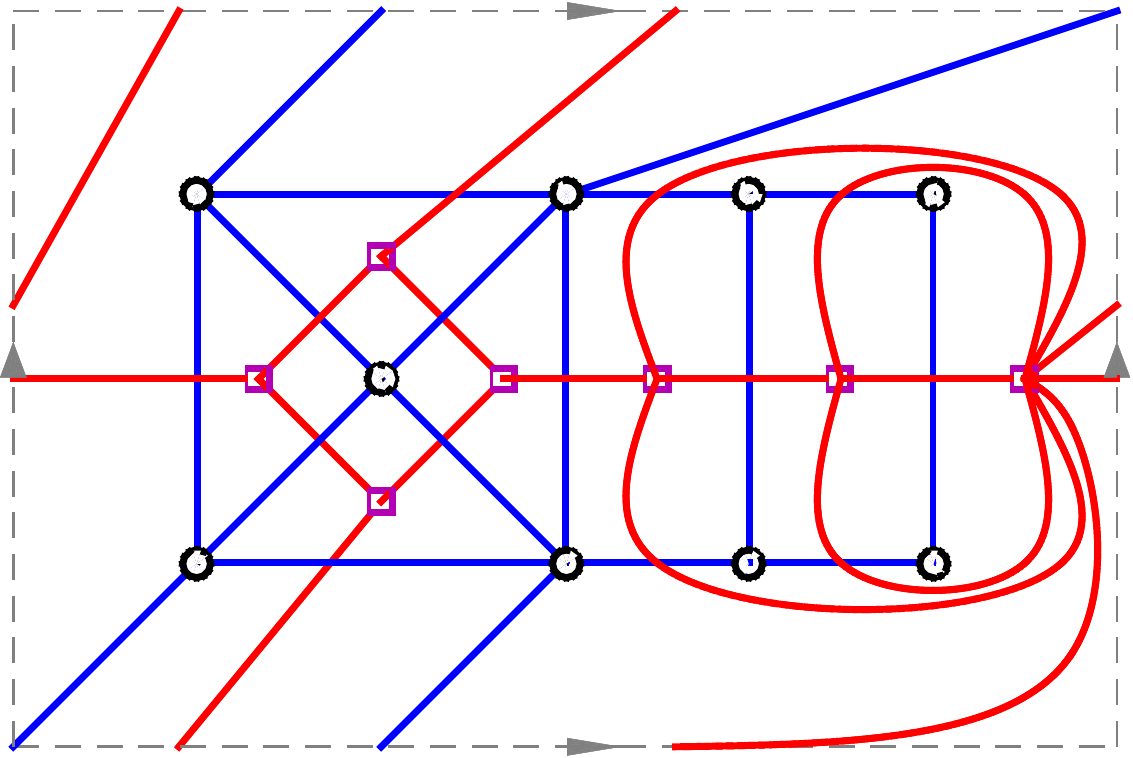}
\includegraphics[height=4cm]{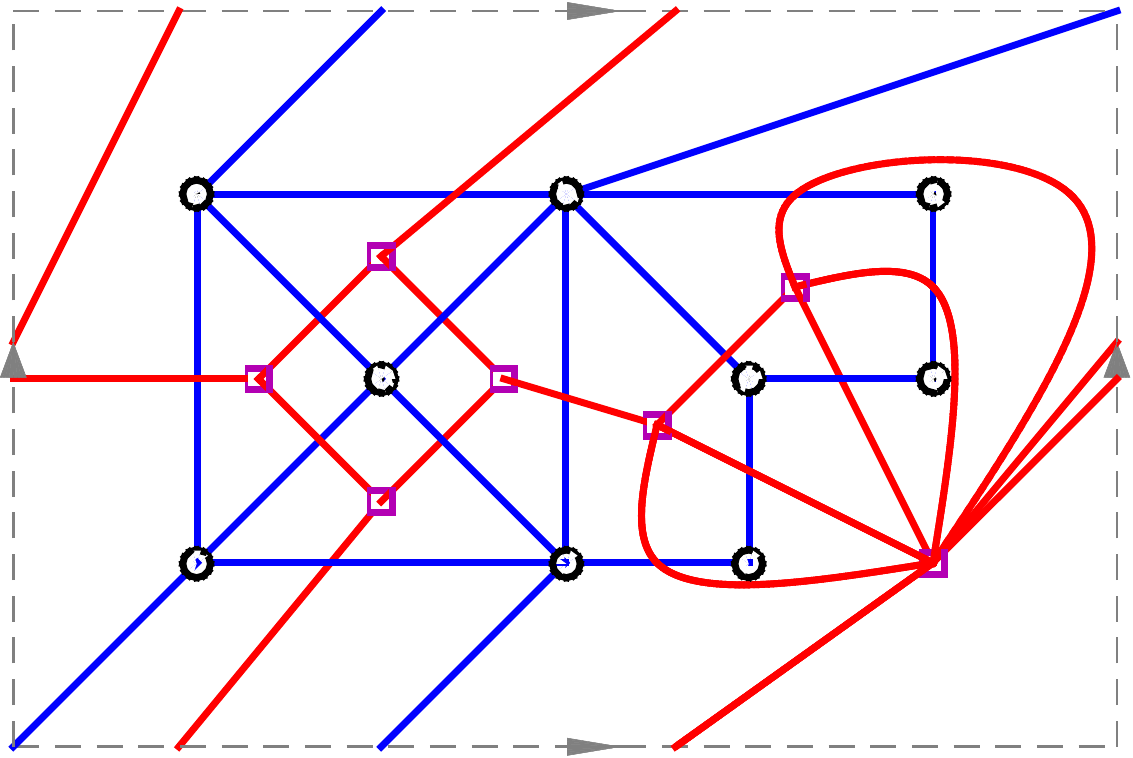}
\caption{A pair of genus one graphs embedded into the torus, with isomorphic dual graphs}
\label{fig:g=1}
\end{figure}

In Figure \ref{fig:g=1}, note that the pair of graphs with isomorphic dual graphs is constructed by gluing two 4-cycles to $K_5$ in a way, such that those cycles lie in the same region of the torus, given by connected components of the complement of the embedded image of $K_5$. In the same way, we can construct examples of graph pairs of higher genus, with embeddings into surface of the corresponding genus, for which the dual graphs are isomorphic, by gluing several cycles to a complete graph $K_n$, or to other types of graphs, in this specific way.

As such, we have tested some higher genus graph pairs, constructed by gluing cycles to a $K_n$, $n=8, \ldots, 40$, and to some other graphs including random graphs. For each pair tested, the quadratic forms represented by the combinatorial Laplacian are globally equivalent. 

\begin{rmk}\label{pL}
The computer experimental results on $p$-adic symbols of the quadratic form for the combinatorial Laplacian of strongly regular graphs are written in Table \ref{tab:p-adic_srg_lap} and Table \ref{tab:p-adic_srg_lap_2}. The classification of ${\rm srg}(25,12,5,6)$ according to the $p$-adic symbols of the combinatorial Laplacian, distinguishes the only graph of 25 vertices with a transitive group, which is known as the Paley graph of order 25 \cite{Paulus}. Again, this classification distinguishes the "Chang graphs" among ${\rm srg}(28,12,6,4)$.
\end{rmk}

\section{A graph distance function}\label{distance}

It is seldom that two big graphs are isomorphic in real applications. Based on the function $\mathcal{T}^G_u$ defined in \eqref{2cor_fun}, we define a function $\mathcal{D}$ to measure the distance between two connected\footnote{A slight variation can accommodate disconnected graphs.} graphs of the same vertex number $|V|$ by 
\begin{equation}\label{eq:distfun}
\mathcal{D}\left({G}, G_1\right) = |V|^2 \sqrt{\frac{\sum\limits_{j=1}^{|V|} \sum\limits_{(x,y)}\left(\operatorname*{\rm sort}\limits_{(x,y)\in V^2}\left( \widetilde{\mathcal{T}}_{u_j}^{G}(x,y)\right) - \operatorname*{\rm sort}\limits_{(x,y)\in V^2}\left( \widetilde{\mathcal{T}}_{u_j}^{G_1}(x,y) \right) \right)^2}{|V|}},
\end{equation}
where\footnote{$\widetilde{\mathcal{T}}_u^G$ differs from $\mathcal{T}_u^G$ only by dropping the term corresponding to the zero eigenvalue, for the purpose of stability of the algorithm.}
\begin{equation}
\widetilde{\mathcal{T}}_u^G(x,y) \equiv \sum_{k=1}^m \frac{t_k(x,y)}{\lambda_k+u}, 
\end{equation}
and $c$ is a small constant that we choose it to be $c=10^{-4}$ for the moment, for initial testing purposes. Suppose $\lambda_1$ of $G$ is less than or equal to the first nonzero eigenvalue of that of the graph $G_1$, then we adopt
$$u_j = \frac{c j}{|V|} \lambda_1,$$
for $j=1, \ldots, |V|$. 

\begin{rmk}
It is clear that the function $\mathcal{D}$ satisfies the triangle inequality 
\[
\mathcal{D}(G, G_1) \leq \mathcal{D}(G, G_2) + \mathcal{D}(G_2, G_1),
\]
when $G$, $G_1$ and $G_2$ are cospectral. In general, when two graphs are not cospectral, some slight modifications of $\mathcal{D}$, e.g. choosing $u_j$ independently of the eigenvalues, can still give distance functions in a strict sense.
\end{rmk}

The basic idea for the function $\mathcal{D}$ is to give a measure of the distance between $\widetilde{\mathcal{T}}_u^G$ and $\widetilde{\mathcal{T}}_u^{G_1}$, as two functions of $u$. Once the functions $\widetilde{\mathcal{T}}_u^G$ and $\widetilde{\mathcal{T}}_u^{G_1}$ are identical, our graph invariant ST for the two graphs are identical, and we regard these two graphs as having zero distance. To compare these two functions, we uniformly sample the points $u_j$ whose values are very small relative to the nonzero eigenvalues, and compare the values of $\widetilde{\mathcal{T}}_{u_j}^G$ and $\widetilde{\mathcal{T}}_{u_j}^{G_1}$, $j=1, \ldots, |V|$. If the functions $\widetilde{\mathcal{T}}_u^G$ and $\widetilde{\mathcal{T}}_u^{G_1}$ are identical at the $|V|$ many points $u_j$, $j=1\ldots, |V|$, and if the two graphs are cospectral, then one sees right away that the two functions are equal. 
Here the ordering of $(x,y)$ is according to the value of $\widetilde{\mathcal{T}}_{u_1}^{G}(x,y)$, and $|V|^2$ is a normalization factor. 

\begin{rmk}
In general, there are other possible choices of sorting rules, such as sorting according to the value of $t_k(x,y)$ without involving any eigenvalues. We may also refine our sorting procedure by first comparing the values of $\widetilde{\mathcal{T}}_{u_1}^{G}(x,y)$, and if some values get too close to each other, we may then compare the corresponding values of $\widetilde{\mathcal{T}}_{u_2}^{G}(x,y)$, and so on. Each different sorting rule may have its advantages in different applications.
\end{rmk}

\begin{rmk}
For consideration in applications, this distance function can be easily extended to apply to pairs of graphs with approximately the same number of vertices, by adding a few isolated vertices.
\end{rmk}

\bibliographystyle{plain}
\bibliography{ref}

\clearpage

\appendix

\section{Initial numerical experiments for the distance function}\label{Experiments}
In the following, we demonstrate several numerical experiments to learn something about the graph distance $\mathcal{D}$ defined in \eqref{eq:distfun}.

\subsection{Dumbbell Graphs}
We first consider the dumbbell graphs, and its one-edge perturbations. The dumbbell graph $DB_n$, shown in Figure \ref{fig:dumbbell}, is constructed by connecting two copies of the complete graph $K_n$ with two edges. Note that there are four types of edges in a dumbbell graph, namely, $e_1$, $e_2$, $e_3$ and $e_4$ in Figure \ref{fig:dumbbell}. Thus, for $i=1,2,3,4$, we compute the distance $\mathcal{D}(DB_n, DB_n-e_i)$, where $DB_n-e_i$ denotes the graph constructed by removing an edge $e_i$ of a dumbbell graph $DB_n$, respectively. 

\begin{figure}[h!]
\centering
\begin{overpic}[height=3cm]{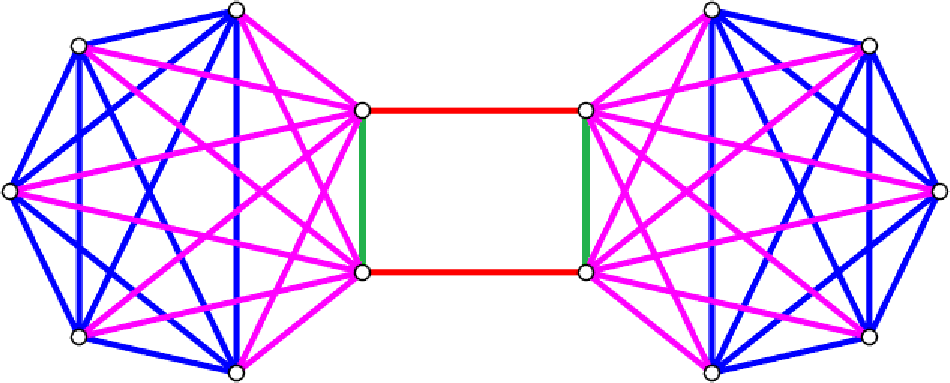}
\put(48,31){$e_1$}
\put(39,19){$e_2$}
\put(33,34){$e_3$}
\put(0,29){$e_4$}
\end{overpic}
\caption{The dumbbell graph}
\label{fig:dumbbell}
\end{figure}
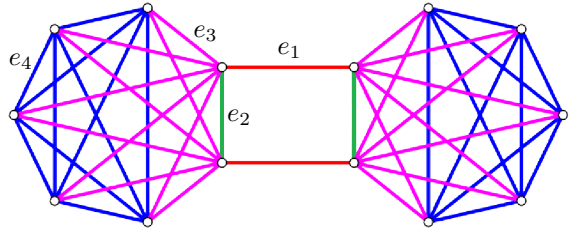

According to the results, shown in Table \ref{tab:dumbbell}, we can easily observe that the graph distance by one-edge perturbations could be huge. For each $n$, as we would expect, the distance $\mathcal{D}(DB_n, DB_n-e_1)$ is significantly larger than others. This agrees with our intuition that the edge $e_1$ is critical in a dumbbell graph. 

\begin{table}[h!]
\centering
\begin{tabular}{c||c|c|c|c}
\hline
\rowcolor{blue!10}
$n$ & {\small $\mathcal{D}(DB_n, DB_n-e_1)$} & {\small $\mathcal{D}(DB_n, DB_n-e_2)$} & {\small $\mathcal{D}(DB_n, DB_n-e_3)$} & {\small$\mathcal{D}(DB_n, DB_n-e_4)$}\\
\hline\hline
5 & $1.7998 \times 10^2$ & $9.1929 \times 10^0$ & $1.5242 \times 10^1$ & $1.2019 \times 10^1$ \\ 
\hline
10 & $1.2061 \times 10^3$ & $8.2997 \times 10^0$ & $1.5786 \times 10^1$ & $1.0000 \times 10^1$ \\ 
\hline
20 & $8.8062 \times 10^3$ & $8.0771 \times 10^0$ & $1.9782 \times 10^1$ & $8.8893 \times 10^0$ \\
\hline
50 & $1.2999 \times 10^5$ & $8.0129 \times 10^0$ & $3.4289 \times 10^1$ & $8.3337 \times 10^0$ \\
\hline
100 & $1.0199 \times 10^6$ & $8.0031 \times 10^0$ & $5.9193 \times 10^1$ & $8.1633 \times 10^0$ \\
\hline
\end{tabular}
\caption{The distance by one-edge perturbations on dumbbell graphs}
\label{tab:dumbbell}
\end{table}

In addition, as $n$ increases, the distance $\mathcal{D}(DB_n, DB_n-e_1)$ increases. This is reasonable and agrees with our expectation. Because the graph would be nearly disconnected if we remove an edge $e_1$. Naturally, the importance of the edge $e_1$ would increase as the number of vertices increases. 

On the other hand, as $n$ increases, the distance $\mathcal{D}(DB_n, DB_n-e_4)$ decreases. Intuitively, this is reasonable since the number of elements in the orbit of $e_4$ under the automorphism group of $DB_n$ increases fast, therefore deleting a single such edge becomes a less significant perturbation. 

Furthermore, as $n$ increases, the distance $\mathcal{D}(DB_n, DB_n-e_3)$ increases. An intuitive explanation is that the number of elements in its orbit under the automorphism group, grows only linearly in $n$, and that the edge $e_3$ connects the important edge $e_1$ and the unimportant edge $e_4$. 

Interestingly, as $n$ increases, the distance $\mathcal{D}(DB_n, DB_n-e_2)$ decreases, which might seem unintuitive. 
To explain this phenomenon, we recall the discussions around Equation \eqref{2} that, difference of values of eigenfunctions at two vertices of an edge, contribute to the difference of two-point correlation functions of the graph with its one-edge perturbation by this edge, and roughly the two-point correlation function is more sensitive to the difference at smaller eigenvalues. Based on this, for each $n$, we check whether for every small eigenvalue $\lambda^{DB_n}$ of the graph $DB_n$, the corresponding eigenfunction $\phi^{DB_n}$ has the property that the value $|\phi^{DB_n}(e_2(1))-\phi^{DB_n}(e_2(2))|$ is relatively small. 

Note that apart from $\lambda^{DB_n}_0=0$, the only small (less than $1$) eigenvalue of a dumbbell graph $DB_n$ is $\lambda^{DB_n}_1$. Others are greater than or equal to $n$. Thus, for each $n$, we compute the value $|\phi^{DB_n}(e_i(1))-\phi^{DB_n}(e_i(2))|$, for $i=1, \ldots, 4$. 
As we expected, the result, shown in Table \ref{tab:dumbbellcheck}, indicates that the value $|\phi^{DB_n}(e_2(1))-\phi^{DB_n}(e_2(2))|$ is indeed relatively small. 

\begin{table}[h!]
\centering
\begin{tabular}{c||c|c|c|c}
\hline
\rowcolor{blue!10}
& \hspace{2cm} & \hspace{2cm} & \hspace{2cm} & \hspace{2cm} \\
\rowcolor{blue!10}
\multirow{-2}{*}{\backslashbox{$n$}{$i$}} & \multirow{-2}{*}{$1$} & \multirow{-2}{*}{$2$} & \multirow{-2}{*}{$3$} & \multirow{-2}{*}{$4$}\\
\hline\hline
5  &  $4.8876 \times 10^{-1}$ & 0 & $1.1179 \times 10^{-1}$ & 0 \\
\hline
10  & $3.8268 \times 10^{-1}$ & 0 & $3.9628 \times 10^{-2}$ & 0 \\
\hline
20  & $2.8978\times 10^{-1}$ & 0 & $1.4623 \times 10^{-2}$ & 0 \\
\hline
50  & $1.9259 \times 10^{-1}$ & 0 & $3.8577 \times 10^{-3}$ & 0 \\
\hline
100 & $1.3870 \times 10^{-1}$ & 0 & $1.3876 \times 10^{-3}$ & 0 \\
\hline
\end{tabular}
\caption{The difference $|\phi^{DB_n}_1(e_i(1)) - \phi^{DB_n}_1(e_i(2))|$.}
\label{tab:dumbbellcheck}
\end{table}

\subsection{Complete Graphs, Cycles and Paths}
 We next consider the complete graphs, which is one of the extreme cases. In the following, we compute the distance $\mathcal{D}(K_n, K_n-e)$, where $K_n-e$ denotes the graph constructed by removing an edge from a complete graph $K_n$. Then we compute the distance between cycles and paths $\mathcal{D}(C_n, P_n)$, where $C_n$ denotes the cycle of $n$ vertices and $P_n$ denotes the path of $n$ vertices. 

According to the result, shown in Table \ref{tab:complete}, we can see that the graph distances by one-edge perturbation on complete graphs are relatively small. From Figure \ref{fig:complete}, it seems that $2$ may be a lower bound for $\mathcal{D}(K_n, K_n-e)$. 

On the other hand, the distances between cycles and paths as their one-edge perturbations are huge, as we would expect, as one is simply connected, while the other one is not. 

\begin{table}[h!]
\begin{tabular}{c||c|c}
\hline
\rowcolor{blue!10}
$n$ & $\mathcal{D}(K_n,K_n-e)$ & \hspace{0.5cm}$\mathcal{D}(C_n, P_n)$\hspace{0.5cm} \\
\hline\hline
5& 3.3330 & $4.6894\times 10^{1}$\\ 
\hline
10& 2.4998 & $7.0100\times 10^{2}$ \\
\hline
20& 2.2221 & $1.1096\times 10^{4}$\\
\hline
50 & 2.0832 & $4.3251\times 10^{5}$\\
\hline
100 & 2.0407 & $6.9186\times 10^{6}$\\
\hline
200 & 2.0201 & $1.1069\times 10^{8}$\\
\hline
500 & 2.0079 & $4.3246\times 10^{9}$\\
\hline
\end{tabular}
\caption{The distance of graphs by one-edge perturbation}
\label{tab:complete}
\end{table}

\begin{figure}[h!]
\centering
\includegraphics[height=6cm]{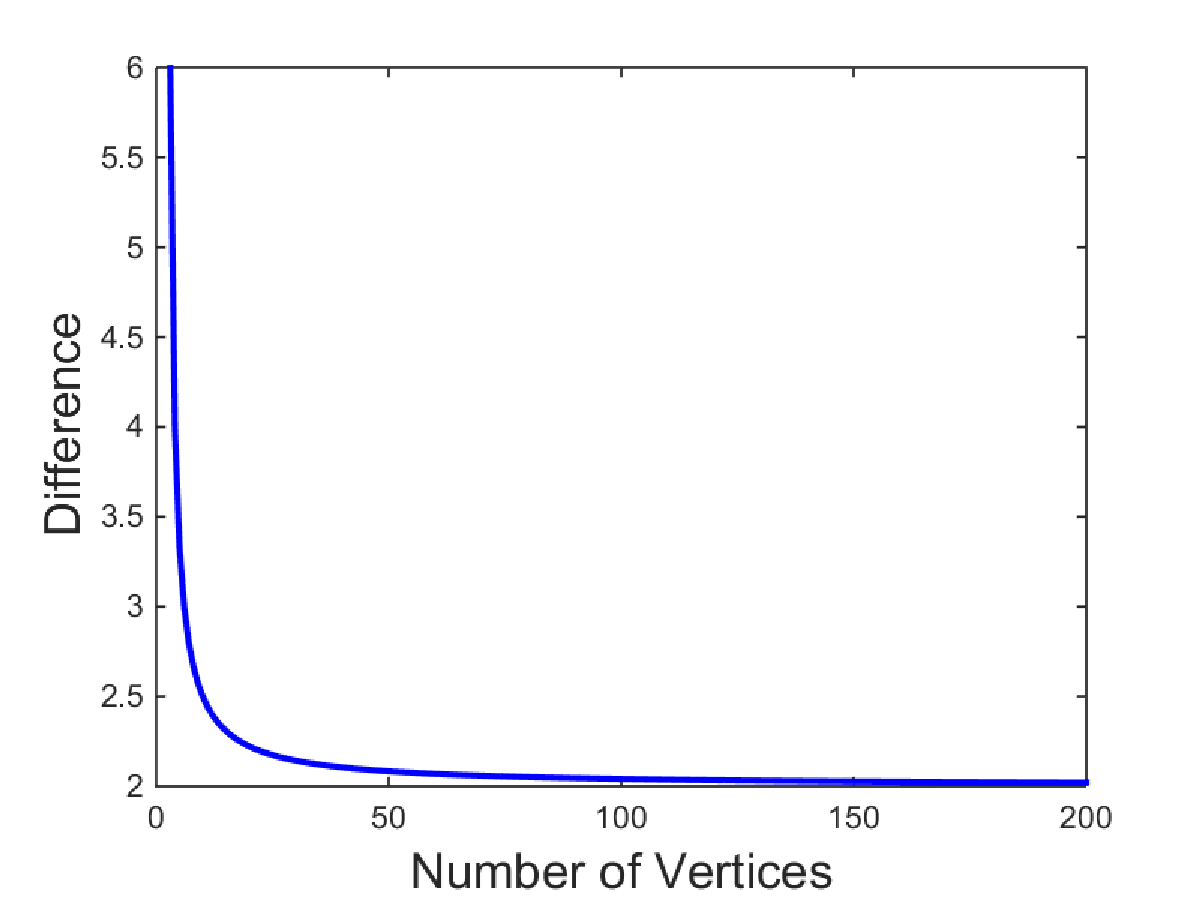}
\caption{The relation between $n$ and $\mathcal{D}(K_n, K_n-e)$}
\label{fig:complete}
\end{figure}

\subsection{Nearly Complete Graphs}
We next consider pairs of nearly complete graphs with the same number of edges. For convenience, we denote the set of all graphs constructed by removing $m$ edges of a complete graph $K_n$ by $\mathscr{K}_n^{(-m)}$. Then we compute the minimal distance among the pairs of graphs in $\mathscr{K}_n^{(-m)}$, namely,  
\[
\min\left\{ \mathcal{D}(G,G_1) \left| G \not\cong G_1, G, G_1 \in \mathscr{K}_n^{(-m)} \right. \right\}, 
\]
for $n = 4, \ldots, 10$, respectively. 

\begin{table}[h!]
\centering
\begin{tabular}{c|cccccc}
\hline
\rowcolor{blue!10} &&&&&&\\
\rowcolor{blue!10} \multirow{-2}{*}{\backslashbox{$n$}{$m$}}& \multirow{-2}{*}{2} & \multirow{-2}{*}{3} & \multirow{-2}{*}{4} & \multirow{-2}{*}{5} & \multirow{-2}{*}{6} & \multirow{-2}{*}{7}\\
\hline\hline
4& 7.7736  & 11.9988    \\ 
\hline
5& 4.1871  & 4.5703  & 4.6263 & 11.3030 \\ 
\hline
6& 3.1366  & 3.1366  & 1.2674 & 2.0128 & 2.9203 & 3.6825 \\ 
\hline
7 & 2.6525 & 2.6525 & $0.8294$ & $0.8417$ & {\color{red} $0.2750$} & {\color{red} $0.5499$} \\
\hline
8 & 2.3777 & 2.3777 & 0.6080 & 0.6080 & 0.1317 & 0.1532 \\
\hline
9 & 2.2015 & 2.2015 & 0.4766 & 0.4766 & 0.0874 & 0.0985\\
\hline
10 & 2.0793  & 2.0793 & 0.3904 & 0.3904 & 0.0621 & 0.0621\\
\hline
\end{tabular}
\caption{The minimal distance among the pairs of graphs in $\mathscr{K}_n^{(-m)}$.}
\label{tab:nearlycomplete}
\end{table}

The result, shown in Table \ref{tab:nearlycomplete}, indicates that the distance between two nonisomorphic graphs measured by $\mathcal{D}$ could be very small. 

Furthermore, in this experiment, we observe that for each pair of graphs above with very small distance (e.g. less than $2$ among these examples), they share the same degree vector of vertices. For instance, the pair of graphs in $\mathscr{K}_6^{(-4)}$ with distance $1.2674$, shown in Figure \ref{fig:nc6(-4)}, are of the same degree of vertices $(3,3,4,4,4,4)$; the pair of graphs in $\mathscr{K}_7^{(-6)}$ with distance $0.2750$, shown in Figure \ref{fig:nc7(-6)}, are of the same degree of vertices $(3,4,4,4,5,5,5)$. 

\begin{figure}[h!]
\centering
\begin{tabular}{cc}
\includegraphics[height=3cm]{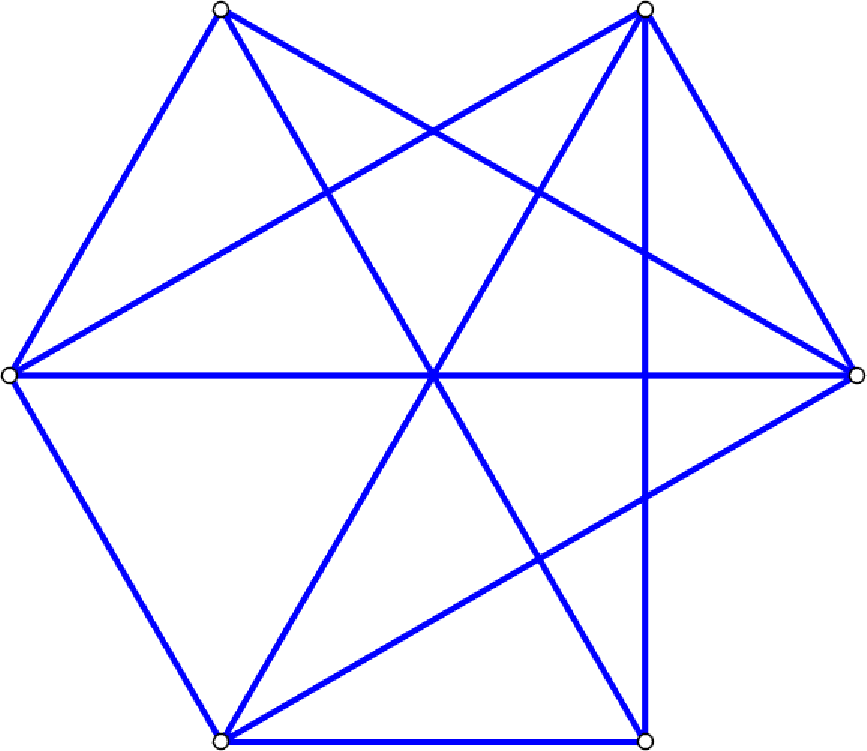} & \includegraphics[height=3cm]{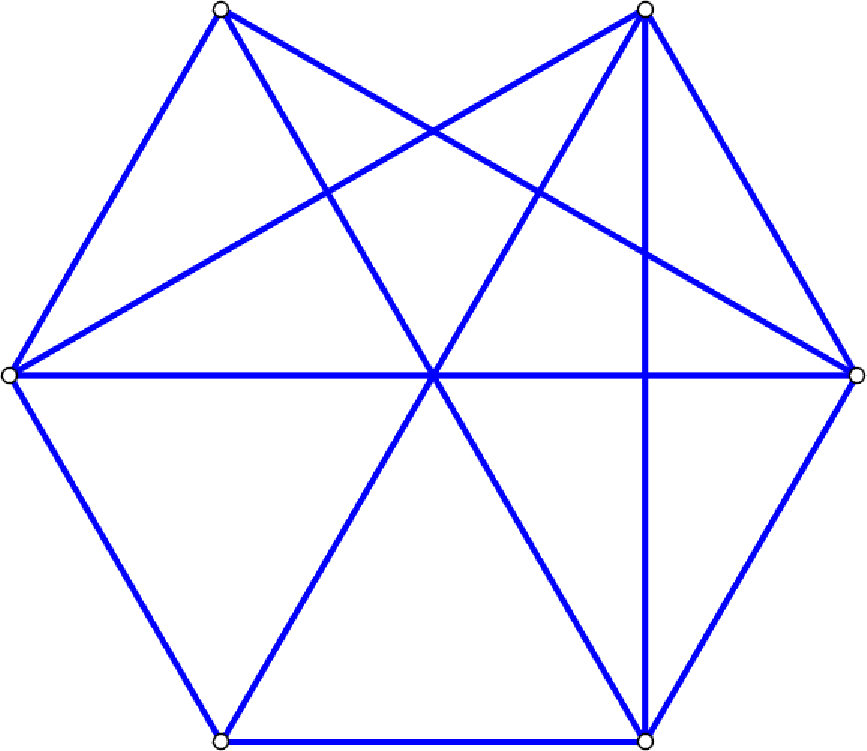} \\
$G\in\mathscr{K}_6^{(-4)}$ & $G_1\in\mathscr{K}_6^{(-4)}$
\end{tabular}
\caption{A pair of graphs in $\mathscr{K}_6^{(-4)}$ with distance $\mathcal{D}(G,G_1)=1.2674$.}
\label{fig:nc6(-4)}
\end{figure}

\begin{figure}[h!]
\centering
\begin{tabular}{cc}
\includegraphics[height=3cm]{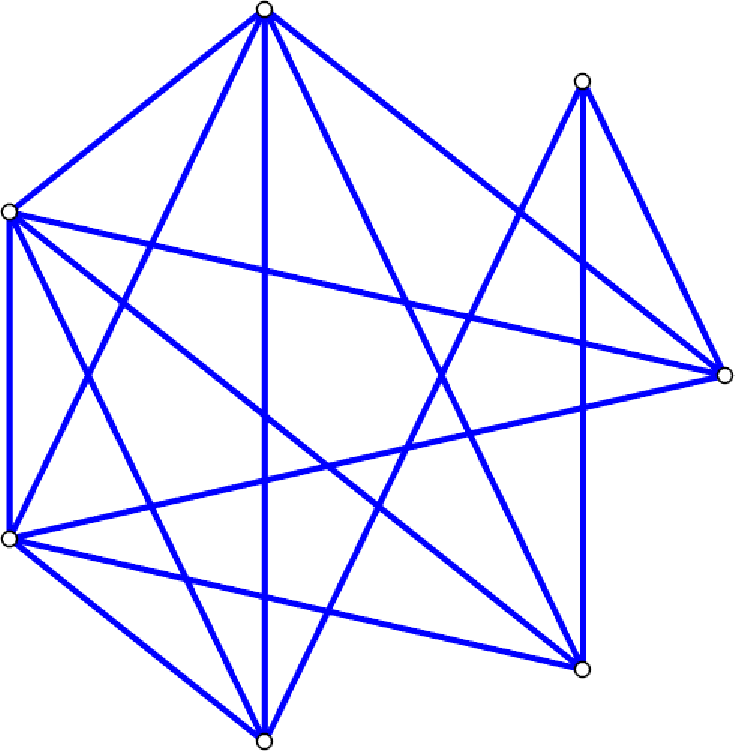} & \includegraphics[height=3cm]{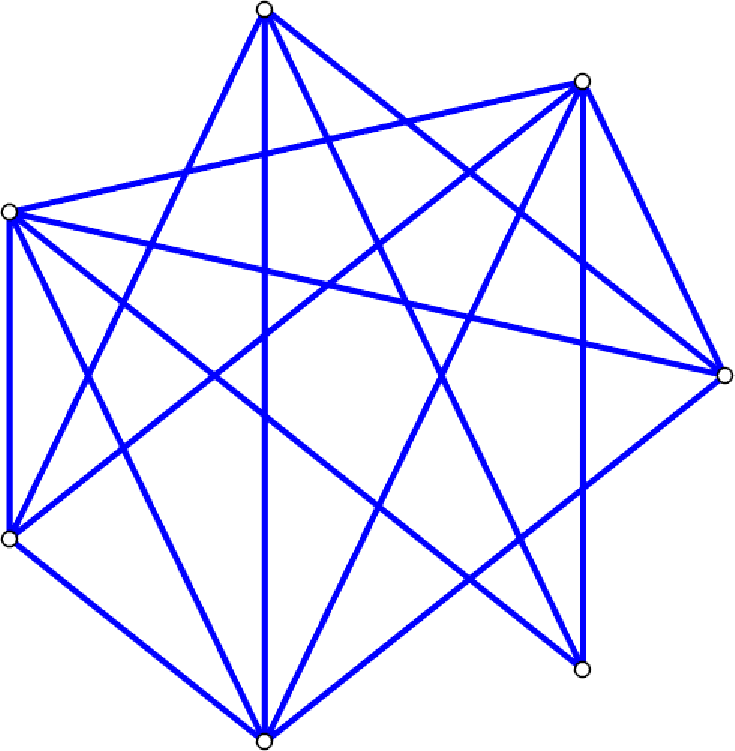} \\
$G\in\mathscr{K}_7^{(-6)}$ & $G_1\in\mathscr{K}_7^{(-6)}$
\end{tabular}
\caption{A pair of (cospectral) graphs in $\mathscr{K}_7^{(-6)}$ with distance $\mathcal{D}(G,G_1)=0.2750$.}
\label{fig:nc7(-6)}
\end{figure}

In addition, the pair of graphs, shown in Figure \ref{fig:nc7(-6)}, are actually cospectral. In Table \ref{tab:nearlycomplete}, we mark the number in red if the pair of graphs with minimal graph distance are actually cospectral. 

For the larger nearly complete graphs, we compute the minimal distance among $10000$ randomly picked nonisomorphic pairs of graphs in $\mathscr{K}_n^{(-m)}$, for $n=50, 100, 200$. The result is shown in Table \ref{tab:nearlycompletelarge}. 

\begin{table}[h!]
\centering
\begin{tabular}{c|cccccccc}
\hline
\rowcolor{blue!10}&&&&&&&&\\
\rowcolor{blue!10} \multirow{-2}{*}{\backslashbox{$n$}{$m$}}& \multirow{-2}{*}{2} & \multirow{-2}{*}{3} & \multirow{-2}{*}{4} & \multirow{-2}{*}{5} & \multirow{-2}{*}{6} & \multirow{-2}{*}{7} &  \multirow{-2}{*}{8} & \multirow{-2}{*}{9} \\
\hline\hline
50 & 1.5057 & 1.5057 & 0.0448 & 0.0448 & 0.0448 & 0.0448 & 0.0011 & 0.0011 \\
\hline
100 & 1.4582 & 1.4582 & 1.4581 & 0.0211 & 0.0211 & 0.0211 & 0.0211 & 0.0211 \\
\hline
200 & 1.4358 & 1.4358 & 1.4358 & 1.4358 & 1.4357 & 1.4357 & 1.4357 & 0.0103\\
\hline
\end{tabular}
\caption{The minimal distance of $10000$ trials among the pairs of graphs in $\mathscr{K}_n^{(-m)}$.}
\label{tab:nearlycompletelarge}
\end{table}

\clearpage

\section{Tables for graphs and quadratic forms}

\begin{table}[h!]
\centering
\begin{tabular}{c|c|c||c|c}
\hline
\rowcolor{blue!10}
{\small Graph Type} & {\small \#Graphs} & ${\small {\rm Eigenvalues}^{\rm Multiplicities}}$ &${\small m}$ & {\small \#Quadratic Forms} \\
\hline\hline
\multirow{3}{*}{srg(16, 6, 2, 2)} & \multirow{3}{*}{2} & \multirow{3}{*}{$(-2)^9, 2^6, 6^1$} & 2 & 2 \\ 
&&& $-2$ &{\color{red} 1}\\
&&& $-6$ &2\\
\hline
\multirow{5}{*}{srg(25, 12, 5, 6)} & \multirow{5}{*}{15} & \multirow{5}{*}{$(-3)^{12}, 2^{12}, 12^1$} & 0 & 15 \\ 
  &  && 2 & 15 \\ 
&&& $3$ & {\color{red} 4} $[10,3,1,1]$\\
&&& $-2$ & {\color{red} 4} $[10,3,1,1]$\\
&&& $-12$ & 15\\
\hline
\multirow{5}{*}{srg(26, 10, 3, 4)} & \multirow{5}{*}{10} & \multirow{5}{*}{$(-3)^{12}, 2^{13}, 10^1$} & 0 & 10 \\ 
  &  && 2 & 10 \\ 
  &  && 3 & {\color{red} 3} $[7,2,1]$ \\ 
  &  && $-2$ & {\color{red} 4} $[5,2,2,1]$ \\ 
  &  && $-10$ & 10 \\ 
\hline
\multirow{4}{*}{srg(28, 12, 6, 4)} & \multirow{4}{*}{4} & \multirow{4}{*}{$(-2)^{20}, 4^7, 12^1$} & 0 & 4 \\ 
 &  && 2 & {\color{red} 2} $[3,1]$ \\ 
 &  && $-4$ & 4 \\ 
 &  && $-12$ & 4 \\ 
\hline
\multirow{3}{*}{srg(29, 14, 6, 7)} & \multirow{3}{*}{41} & \multirow{3}{*}{$(\frac{-1 \pm \sqrt{29}}{2})^{14}, 14^1$} & 0 & 41 \\ 
 &  && 2 & 41 \\ 
 &  && $-14$ & 41 \\ 
\hline
srg(35, 18, 9, 9) & 3854 & $(-3)^{20}, 3^{14}, 18^1$ & 2 & 3854 \\ 
\hline
\multirow{5}{*}{srg(36, 14, 4, 6)} & \multirow{5}{*}{180} & \multirow{5}{*}{$(-4)^{14}, 2^{21}, 14^1$} & 0 & 180 \\
 &  && 2 & 180 \\ 
 &  && 4 & {\color{red} 155} \\ 
 &  && $-2$ & {\color{red} 9} $[66,44,43,11,9,2,2,2,1]$  \\ 
 &  && $-14$ & 180 \\ 
\hline
\multirow{5}{*}{srg(40, 12, 2, 4)} & \multirow{5}{*}{28} & \multirow{5}{*}{$(-4)^{15}, 2^{24}, 12^1$} & 0 & 28 \\ 
 &  && 2 & 28 \\ 
 &  && 4 & 28  \\ 
 &  && $-2$ & {\color{red} 6} $[13,8,3,2,1,1]$ \\ 
 &  && $-12$ & 28  \\ 
\hline
\multirow{4}{*}{srg(45, 12, 3, 3)} & \multirow{4}{*}{78} & \multirow{4}{*}{$(-3)^{24}, 3^{20}, 12^1$} & 2 & 78 \\ 
 &  && 3 & {\color{red} 21}  \\ 
 &  && $-3$ & {\color{red} 76} \\ 
 &  && $-12$ & 78 \\ 
\hline
\multirow{5}{*}{srg(50, 21, 8, 9)} & \multirow{5}{*}{18} & \multirow{5}{*}{$(-4)^{24}, 3^{25}, 21^1$} & 0 & 18 \\ 
 &  && 2 & 18 \\ 
 &  && 4 & {\color{red} 17}\\ 
 &  && $-3$ & {\color{red} 5} $[10,4,2,1,1]$ \\ 
 &  && $-21$ & 18 \\ 
\hline
\multirow{4}{*}{srg(64, 18, 2, 6)} & \multirow{4}{*}{167} & \multirow{4}{*}{$(-6)^{18}, 2^{45}, 18^1$} & 0 & 167 \\ 
 &  && 6 & 167 \\ 
 &  && $-2$ & {\color{red} 4} $[156,9,1,1]$ \\ 
 &  && $-18$ & 167 \\ 
\hline
\end{tabular}
\caption{The number of distinct quadratic forms $(A+mI)^2$ with respect to ${\rm srg}(n,k,\lambda,\mu)$.}
\label{tab:srg}
\end{table}
\clearpage

\begin{table}
\centering
\begin{tabular}{c||c|l|c}
\hline
\rowcolor{blue!10}
{\small Graph Type}  & $p$ & $p$-Adic Symbols & \#Graphs \\
\hline\hline
\multirow{2}{*}{srg(16, 6, 2, 2)}  & 3  & $1^{10-} 3^{6+}$ & \multirow{2}{*}{2} \\  
 & 7 & $1^{15-} 7^{1+}$ \\
\hline
\multirow{4}{*}{srg(25, 12, 5, 6)} & 3 & $1^{13+} 3^{12+}$ & \multirow{2}{*}{10} \\
  & 13 & $1^{24+} 13^{1+}$ \\ 
\cline{2-4}
 & 3 & $1^{13-} 3^{12-}$ & \multirow{2}{*}{5} \\
 &13 & $1^{24+} 13^{1+}$ \\
\hline
\multirow{4}{*}{srg(26, 10, 3, 4)}  & 3 & $1^{13-} 3^{13+}$ & \multirow{2}{*}{3} \\ 
  & 11  & $1^{25+} 11^{1+}$ \\ 
\cline{2-4}
  & 3  & $1^{13+} 3^{13-}$ & \multirow{2}{*}{7} \\ 
  & 11  & $1^{25+} 11^{1+}$ \\ 
\hline
\multirow{2}{*}{srg(28, 12, 6, 4)}  & 5 & $1^{21+} 5^{7-}$ & \multirow{2}{*}{4}\\ 
   & 13 & $1^{27+} 13^{1-}$ \\ 
\hline
\multirow{3}{*}{srg(29, 14, 6, 7)}  & 3 & $1^{28-} 3^{1+}$ & \multirow{3}{*}{41} \\ 
 & 5 & $1^{28+} 5^{1-}$  \\ 
 & 7 & $1^{15-} 7^{14-}$ \\ 
\hline
\multirow{1}{*}{srg(35, 18, 9, 9)}  & 19 & $1^{34+} 19^{1+}$ & \multirow{1}{*}{3854} \\ 
\hline
\multirow{8}{*}{srg(36, 14, 4, 6)} & 3 & $1^{14-} 3^{8-} 9^{14-}$ & \multirow{2}{*}{132} \\
& 5 & $1^{35-} 5^{1-}$ \\
\cline{2-4}
& 3 & $1^{14+} 3^{8-} 9^{14+}$ & \multirow{2}{*}{44} \\
& 5 & $1^{35-} 5^{1-}$  \\
\cline{2-4}
& 3 & $1^{13+} 3^{10+} 9^{13-}$ & \multirow{2}{*}{2} \\
& 5 & $1^{35-} 5^{1-}$  \\
\cline{2-4}
& 3 & $1^{12-} 3^{12-} 9^{12-}$ & \multirow{2}{*}{2} \\
& 5 & $1^{35-} 5^{1-}$  \\
\hline
\multirow{8}{*}{srg(40, 12, 2, 4)} & 3 & $1^{15-} 3^{11-} 9^{14-}$ & \multirow{2}{*}{11} \\ 
& 13 & $1^{39+} 13^{1+}$   \\ 
\cline{2-4}
& 3 & $1^{15+} 3^{11-} 9^{14+}$ & \multirow{2}{*}{13} \\ 
& 13 & $1^{39+} 13^{1+}$   \\ 
\cline{2-4}
& 3 & $1^{13-} 3^{15-} 9^{12-}$ & \multirow{2}{*}{3} \\ 
& 13 & $1^{39+} 13^{1+}$   \\ 
\cline{2-4}
& 3 & $1^{11+} 3^{19-} 9^{10+}$ & \multirow{2}{*}{1} \\ 
& 13 & $1^{39+} 13^{1+}$   \\ 
\hline
srg(45, 12, 3, 3) & 13 & $1^{44-} 13^{1-}$ & 78  \\ 
\hline
\multirow{2}{*}{srg(50, 21, 8, 9)} & 3 & $1^{26+} 3^{24+}$ & \multirow{2}{*}{18} \\
& 11 & $1^{49-} 11^{1+}$ \\ 
\hline
\multirow{3}{*}{srg(64, 18, 2, 6)} & 3 & $1^{19+} 3^{45+}$ & \multirow{3}{*}{167} \\ 
& 5 & $1^{46-} 5^{18+}$\\
& 19 & $1^{63-} 19^{1+}$\\
\hline
\end{tabular}
\caption{The $p$-adic symbols for quadratic form $A+I$ with respect to ${\rm srg}(n,k,\lambda,\mu)$.}
\label{tab:p-adic}
\end{table}
\clearpage

\begin{table}
\centering
\begin{tabular}{c||c|l|c}
\hline
\rowcolor{blue!10}
{\small Graph Type}  & $p$ & $p$-Adic Symbols & \#Graphs \\
\hline\hline
\multirow{4}{*}{srg(16, 6, 2, 2)}  & 2 & $[0, 6, 7, 0, 0], [1, 4, 1, 0, 0], [2, 6, 3, 1, 4]$ & \multirow{2}{*}{1}  \\  
& 3 & $[0, 15, 1], [1, 1, -1]$ \\
\cline{2-4}
& 2 & $[0, 6, 3, 0, 0], [1, 4, 3, 1, 2], [2, 6, 5, 1, 6]$ & \multirow{2}{*}{1} \\
& 3 & $[0, 15, 1], [1, 1, -1]$ \\
\cline{2-4}
\hline
\multirow{4}{*}{srg(25, 12, 5, 6)} & 2 & $[0, 12, 1, 0, 0], [1, 12, 1, 0, 0], [2, 1, 3, 1, 3]$ & \multirow{2}{*}{10} \\
& 3 & $[0, 12, 1], [1, 13, 1]$ \\
\cline{2-4}
& 2 & $[0, 12, 5, 0, 0], [1, 12, 5, 0, 0], [2, 1, 3, 1, 3]$ & \multirow{2}{*}{5} \\
& 3 & $[0, 12, -1], [1, 13, -1]$ \\
\hline
\multirow{12}{*}{srg(26, 10, 3, 4)}  & 2 & $[0, 12, 1, 0, 0], [1, 14, 5, 1, 6]$ & \multirow{3}{*}{5} \\ 
& 3 & $[0, 14, -1], [1, 12, 1]$ \\
& 5 & $[0, 25, -1], [1, 1, -1]$ \\
\cline{2-4}
& 2 & $[0, 12, 5, 0, 0], [1, 14, 1, 1, 6]$ & \multirow{3}{*}{2} \\
& 3 & $[0, 14, 1], [1, 12, -1]$ \\
& 5 & $[0, 25, -1], [1, 1, -1]$ \\
\cline{2-4}
& 2 & $[0, 12, 5, 0, 0], [1, 14, 1, 1, 2]$ & \multirow{3}{*}{2} \\
& 3 & $[0, 14, -1], [1, 12, 1]$ \\
& 5 & $[0, 25, -1], [1, 1, -1]$ \\
\cline{2-4}
& 2 & $[0, 12, 1, 0, 0], [1, 14, 5, 1, 2]$ & \multirow{3}{*}{1} \\
& 3 & $[0, 14, 1], [1, 12, -1]$ \\
& 5 & $[0, 25, -1], [1, 1, -1]$ \\
\hline
\multirow{4}{*}{srg(28, 12, 6, 4)} & 2 & $[0, 8, 5, 0, 0], [1, 12, 5, 0, 0], [3, 8, 3, 1, 6]$ & \multirow{2}{*}{3} \\ 
& 3 & $[0, 27, 1], [1, 1, 1]$ \\
\cline{2-4}
& 2 & $[0, 6, 3, 0, 0], [1, 15, 7, 1, 5], [3, 7, 7, 1, 1]$ & \multirow{2}{*}{1} \\
& 3 & $[0, 27, 1], [1, 1, 1]$ \\
\hline
\multirow{2}{*}{srg(29, 14, 6, 7)} & 2 & $[0, 28, 1, 0, 0], [1, 1, 7, 1, 7]$ & \multirow{2}{*}{41} \\ 
& 7 & $[0, 14, -1], [1, 15, -1]$ \\
\hline
\multirow{6}{*}{srg(35, 18, 9, 9)} & 2 & $[0, 34, 7, 0, 0], [1, 1, 7, 1, 7]$ & \multirow{2}{*}{3816} \\
& 3 & $[0, 13, -1], [1, 8, -1], [2, 14, -1]$ \\ 
\cline{2-4}
& 2 & $[0, 34, 7, 0, 0], [1, 1, 7, 1, 7]$ & \multirow{2}{*}{37} \\
& 3 & $[0, 13, 1], [1, 8, -1], [2, 14, 1]$ \\
\cline{2-4}
& 2 & $[0, 34, 7, 0, 0], [1, 1, 7, 1, 7]$ & \multirow{2}{*}{1} \\
& 3 & $[0, 11, -1], [1, 12, -1], [2, 12, -1]$ \\
\hline
\multirow{12}{*}{srg(36, 14, 4, 6)} & 2 & $[0, 14, 7, 0, 0], [1, 8, 5, 0, 0], [3, 14, 5, 1, 6]$ & \multirow{2}{*}{109}  \\ 
& 7 & $[0, 35, 1], [1, 1, 1]$ \\
\cline{2-4}
& 2 & $[0, 14, 3, 0, 0], [1, 8, 1, 0, 0], [3, 14, 5, 1, 2]$ & \multirow{2}{*}{48}  \\ 
& 7 & $[0, 35, 1], [1, 1, 1]$ \\
\cline{2-4}
& 2 & $[0, 12, 1, 0, 0], [1, 10, 3, 0, 0], [2, 2, 7, 0, 0], [3, 12, 3, 1, 6]$ & \multirow{2}{*}{19}  \\
& 7 & $[0, 35, 1], [1, 1, 1]$ \\
\cline{2-4}
& 2 & $[0, 12, 1, 0, 0], [1, 10, 3, 0, 0], [2, 2, 3, 0, 0], [3, 12, 7, 1, 2]$ &  \multirow{2}{*}{1} \\
& 7 & $[0, 35, 1], [1, 1, 1]$ \\
\cline{2-4}
& 2 & $[0, 10, 7, 0, 0], [1, 12, 5, 0, 0], [2, 4, 1, 0, 0], [3, 10, 5, 1, 6]$ & \multirow{2}{*}{2}  \\
& 7 & $[0, 35, 1], [1, 1, 1]$ \\
\cline{2-4}
& 2 & $[0, 8, 1, 0, 0], [1, 14, 3, 0, 0], [2, 6, 7, 0, 0], [3, 8, 3, 1, 6]$ & \multirow{2}{*}{1}  \\
& 7 & $[0, 35, 1], [1, 1, 1]$ \\
\hline
\end{tabular}
\caption{The $p$-adic symbols for quadratic form $A$ with respect to ${\rm srg}(n,k,\lambda,\mu)$ by using Sage.}
\label{tab:p-adic_srg}
\end{table}

\clearpage

\begin{table}
\centering
\begin{tabular}{c||c|l|c}
\hline
\rowcolor{blue!10}
{\small Graph Type}  & $p$ & $p$-Adic Symbols & \#Graphs \\
\hline\hline
\multirow{10}{*}{srg(40, 12, 2, 4)} & 2 & $[0, 16, 5, 0, 0], [1, 8, 1, 0, 0], [3, 16, 1, 1, 4]$ & \multirow{2}{*}{17} \\
& 3 & $[0, 39, -1], [1, 1, 1]$ \\ 
\cline{2-4}
& 2 & $[0, 14, 7, 0, 0], [1, 10, 3, 0, 0], [2, 2, 3, 0, 0], [3, 14, 3, 1, 4]$ & \multirow{2}{*}{7} \\
& 3 & $[0, 39, -1], [1, 1, 1]$ \\ 
\cline{2-4}
& 2 & $[0, 12, 1, 0, 0], [1, 12, 5, 0, 0], [2, 4, 5, 0, 0], [3, 12, 5, 1, 4]$ & \multirow{2}{*}{2} \\
& 3 & $[0, 39, -1], [1, 1, 1]$ \\ 
\cline{2-4}
& 2 & $[0, 14, 7, 0, 0], [1, 10, 3, 0, 0], [2, 2, 7, 0, 0], [3, 14, 7, 1, 0]$ & \multirow{2}{*}{1} \\
& 3 & $[0, 39, -1], [1, 1, 1]$ \\ 
\cline{2-4}
& 2 & $[0, 10, 7, 0, 0], [1, 14, 3, 0, 0], [2, 6, 3, 0, 0], [3, 10, 3, 1, 4]$ & \multirow{2}{*}{1} \\
& 3 & $[0, 39, -1], [1, 1, 1]$ \\ 
\hline
\multirow{18}{*}{srg(45, 12, 3, 3)} & 2 & $[0, 44, 5, 0, 0], [2, 1, 7, 1, 7]$ & \multirow{2}{*}{19} \\
& 3 & $[0, 20, 1], [1, 5, 1], [2, 20, 1]$ \\
\cline{2-4}
& 2 & $[0, 44, 5, 0, 0], [2, 1, 7, 1, 7]$ & \multirow{2}{*}{18} \\
& 3 & $[0, 19, 1], [1, 7, -1], [2, 19, -1]$ \\
\cline{2-4}
& 2 & $[0, 44, 5, 0, 0], [2, 1, 7, 1, 7]$ & \multirow{2}{*}{13} \\
& 3 & $[0, 20, -1], [1, 5, 1], [2, 20, -1]$ \\
\cline{2-4}
& 2 & $[0, 44, 5, 0, 0], [2, 1, 7, 1, 7]$ & \multirow{2}{*}{8} \\
& 3 & $[0, 19, -1], [1, 7, -1], [2, 19, 1]$ \\
\cline{2-4}
& 2 & $[0, 44, 5, 0, 0], [2, 1, 7, 1, 7]$ & \multirow{2}{*}{6} \\
& 3 & $[0, 18, 1], [1, 9, 1], [2, 18, 1]$ \\
\cline{2-4}
& 2 & $[0, 44, 5, 0, 0], [2, 1, 7, 1, 7]$ & \multirow{2}{*}{6} \\
& 3 & $[0, 17, -1], [1, 11, -1], [2, 17, 1]$ \\
\cline{2-4}
& 2 & $[0, 44, 5, 0, 0], [2, 1, 7, 1, 7]$ & \multirow{2}{*}{3} \\
& 3 & $[0, 18, -1], [1, 9, 1], [2, 18, -1]$ \\
\cline{2-4}
& 2 & $[0, 44, 5, 0, 0], [2, 1, 7, 1, 7]$ & \multirow{2}{*}{3} \\
& 3 & $[0, 17, 1], [1, 11, -1], [2, 17, -1]$ \\
\cline{2-4}
& 2 & $[0, 44, 5, 0, 0], [2, 1, 7, 1, 7]$ & \multirow{2}{*}{2} \\
& 3 & $[0, 15, 1], [1, 15, -1], [2, 15, -1]$ \\
\hline
\multirow{3}{*}{srg(50, 21, 8, 9)} & 2 & $[0, 26, 7, 0, 0], [2, 24, 1, 0, 0]$ & \multirow{3}{*}{18}  \\
& 3 & $[0, 24, 1], [1, 26, 1]$ \\
& 7 & $[0, 49, -1], [1, 1, -1]$ \\
\hline
\multirow{8}{*}{srg(64, 18, 2, 6)} & 2 & $[0, 18, 7, 0, 0], [1, 28, 5, 0, 0], [2, 18, 3, 1, 4]$ & \multirow{2}{*}{155} \\
& 3 & $[0, 45, -1], [1, 18, 1], [2, 1, -1]$ \\ 
\cline{2-4}
& 2 & $[0, 16, 1, 0, 0], [1, 32, 5, 0, 0], [2, 16, 5, 1, 4]$ & \multirow{2}{*}{10} \\
& 3 & $[0, 45, -1], [1, 18, 1], [2, 1, -1]$ \\
\cline{2-4}
& 2 & $[0, 14, 7, 0, 0], [1, 36, 5, 0, 0], [2, 14, 3, 1, 4]$ & \multirow{2}{*}{1} \\
& 3 & $[0, 45, -1], [1, 18, 1], [2, 1, -1]$ \\
\cline{2-4}
& 2 & $[0, 18, 7, 0, 0], [1, 28, 1, 0, 0], [2, 18, 7, 1, 0]$ & \multirow{2}{*}{1} \\
& 3 & $[0, 45, -1], [1, 18, 1], [2, 1, -1]$  \\
\hline
\end{tabular}
\caption{The $p$-adic symbols for quadratic form $A$ with respect to ${\rm srg}(n,k,\lambda,\mu)$ by using Sage.}
\label{tab:p-adic_srg2}
\end{table}
\clearpage

\begin{figure}
\centering
\includegraphics[height=6cm]{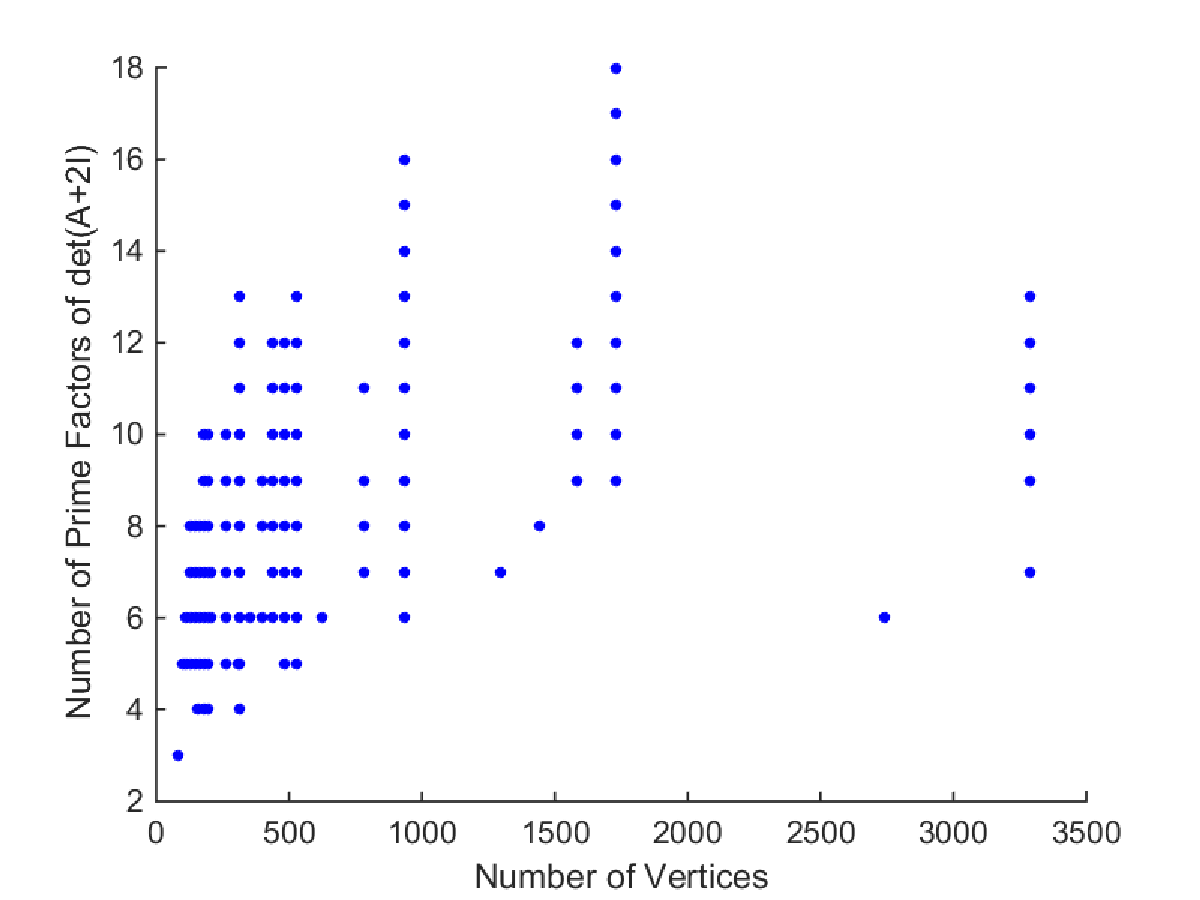}
\caption{The number of distinct prime factors of $\det(A+2I)$}
\label{fig:CFI}
\end{figure}
$~$

\begin{table}
\centering
\begin{tabular}{c|c||c|l}
\hline
\rowcolor{blue!10}
{\small Graph Type}  & \#Vertices & $p$ & $p$-Adic Symbols \\
\hline\hline
$\multirow{4}{*}{{\rm CFI}({\rm reg}(4,3))}$ & \multirow{4}{*}{40} & {\color{red} 2} & $[0, 30, 7, 0, 0], [1, 4, 1, {\color{red}0}, 0], [2, 6, 3, 1, 4]$  \\ 
& & 3 & $[0, 39, -1], [1, 1, 1]$ \\
\cline{3-4}
& & {\color{red} 2} & $[0, 30, 7, 0, 0], [1, 4, 1, {\color{red}1}, 0], [2, 6, 3, 1, 4]$  \\ 
& & 3 & $[0, 39, -1], [1, 1, 1]$ \\
\hline
$\multirow{6}{*}{{\rm CFI}({\rm reg}(8,3))}$ & \multirow{6}{*}{80} & {\color{red} 2} & $[0, 62, 7, 0, 0], [1, 6, 7, {\color{red}0}, 0], [2, 10, 7, 0, 0], [3, 2, 7, 0, 0]$  \\ 
& & 3 & $[0, 79, -1], [1, 1, -1]$ \\
& & 5 & $[0, 79, 1], [1, 1, -1]$ \\
\cline{3-4}
& & {\color{red} 2} & $[0, 62, 7, 0, 0], [1, 6, 7, {\color{red}1}, 0], [2, 10, 7, 0, 0], [3, 2, 7, 0, 0]$ \\
& & 3 & $[0, 79, -1], [1, 1, -1]$ \\
& & 5 & $[0, 79, 1], [1, 1, -1]$ \\
\hline
$\multirow{4}{*}{{\rm CFI}({\rm reg}(8,3))}$ & \multirow{4}{*}{80} & {\color{red} 2} &  $[0, 62, 7, 0, 0], [1, 8, 1, {\color{red}0}, 0], [2, 6, 7, 0, 0], [3, 4, 1, 0, 0]$  \\ 
& & 3 & $[0, 78, -1], [1, 2, -1]$ \\
\cline{3-4}
& & {\color{red} 2} & $[0, 62, 7, 0, 0], [1, 8, 1, {\color{red}1}, 0], [2, 6, 7, 0, 0], [3, 4, 1, 0, 0]$ \\
& & 3 & $[0, 78, -1], [1, 2, -1]$ \\
\hline
$\multirow{4}{*}{{\rm CFI}({\rm reg}(8,3))}$ & \multirow{4}{*}{80} & {\color{red} 2} &  $[0, {\color{red}60}, {\color{red}1}, 0, 0], [1, {\color{red}8}, 1, 0, 0], [2, {\color{red}12}, {\color{red}1}, 0, 0]$  \\ 
& & 3 & $[0, 78, -1], [1, 2, -1]$ \\
\cline{3-4}
& & {\color{red} 2} & $[0, {\color{red}62}, {\color{red}7}, 0, 0], [1, {\color{red}4}, 1, 0, 0], [2, {\color{red}14}, {\color{red}7}, 0, 0]$ \\
& & 3 & $[0, 78, -1], [1, 2, -1]$ \\
\hline
$\multirow{4}{*}{{\rm CFI}({\rm reg}(8,3))}$ & \multirow{4}{*}{80} 
& {\color{red}2} &  $[0, 62, 7, 0, 0], [1, 8, 1, {\color{red}0}, 0], [2, 6, 7, 0, 0], [3, 4, 5, 0, 0]$  \\ 
& & 3 & $[0, 79, 1], [1, 1, -1]$ \\
\cline{3-4}
& & {\color{red}2} & $[0, 62, 7, 0, 0], [1, 8, 1, {\color{red}1}, 0], [2, 6, 7, 0, 0], [3, 4, 5, 0, 0]$ \\
& & 3 & $[0, 79, 1], [1, 1, -1]$ \\
\hline
\end{tabular}
\caption{The $p$-adic symbols for quadratic form $A$ with respect to ${\rm CFI}({\rm reg}(n,3))$ by using Sage.}
\label{tab:p-adic_CFI}
\end{table}
\clearpage

\begin{table}
\centering
\begin{tabular}{c|c||c|l}
\hline
\rowcolor{blue!10}
{\small Graph Type}  & \#Vertices & $p$ & $p$-Adic Symbols \\
\hline\hline
$\multirow{4}{*}{{\rm CFI}({\rm reg}(10,3))}$ & \multirow{4}{*}{100} 
& {\color{red}2} &  $[0, 78, 7, 0, 0], [1, 10, 3, {\color{red}1}, 0], [2, 6, 7, 1, 0], [3, 6, 7, 0, 0]$  \\ 
& & 3 & $[0, 99, -1], [1, 1, 1]$ \\
\cline{3-4}
& & {\color{red}2} & $[0, 78, 7, 0, 0], [1, 10, 3, {\color{red}0}, 0], [2, 6, 7, 1, 0], [3, 6, 7, 0, 0]$ \\
& & 3 & $[0, 99, -1], [1, 1, 1]$ \\
\hline
$\multirow{4}{*}{{\rm CFI}({\rm reg}(10,3))}$ & \multirow{4}{*}{100} 
& {\color{red}2} &  $[0, 78, 7, 0, 0], [1, 12, 1, {\color{red}1}, 0], [2, 2, 7, 0, 0], [3, 8, 5, 0, 0]$  \\ 
& & 3 & $[0, 99, -1], [1, 1, 1]$ \\
\cline{3-4}
& & {\color{red}2} & $[0, 78, 7, 0, 0], [1, 12, 1, {\color{red}0}, 0], [2, 2, 7, 0, 0], [3, 8, 5, 0, 0]$ \\
& & 3 & $[0, 99, -1], [1, 1, 1]$ \\
\hline
$\multirow{6}{*}{{\rm CFI}({\rm reg}(10,3))}$ & \multirow{6}{*}{100} 
& {\color{red}2} &  $[0, 78, 7, 0, 0], [1, 12, 5, {\color{red}1}, 0], [2, 2, 7, 0, 0], [3, 8, 1, 0, 0]$  \\ 
& & 3 & $[0, 99, 1], [1, 1, 1]$ \\
& & 7 & $[0, 99, 1], [1, 1, -1]$ \\
\cline{3-4}
& & {\color{red}2} & $[0, 78, 7, 0, 0], [1, 12, 5, {\color{red}0}, 0], [2, 2, 7, 0, 0], [3, 8, 1, 0, 0]$ \\
& & 3 & $[0, 99, 1], [1, 1, 1]$ \\
& & 7 & $[0, 99, 1], [1, 1, -1]$ \\
\hline
$\multirow{4}{*}{{\rm CFI}({\rm reg}(10,3))}$ & \multirow{4}{*}{100} 
& {\color{red}2} &  $[0, 78, 7, 0, 0], [1, 8, 1, {\color{red}1}, 0], [2, 10, 7, 0, 0], [4, 4, 3, 1, 2]$  \\ 
& & 3 & $[0, 99, 1], [1, 1, 1]$ \\
\cline{3-4}
& & {\color{red}2} & $[0, 78, 7, 0, 0], [1, 8, 1, {\color{red}0}, 0], [2, 10, 7, 0, 0], [4, 4, 3, 1, 2]$ \\
& & 3 & $[0, 99, 1], [1, 1, 1]$ \\
\hline
$\multirow{4}{*}{{\rm CFI}({\rm reg}(10,3))}$ & \multirow{4}{*}{100} 
& {\color{red}2} &  $[0, 78, 7, 0, 0], [1, 10, 7, {\color{red}1}, 0], [2, 6, 3, 1, 4], [3, 6, 7, 0, 0]$  \\ 
& & 3 & $[0, 99, -1], [1, 1, 1]$ \\
\cline{3-4}
& & {\color{red}2} & $[0, 78, 7, 0, 0], [1, 10, 7, {\color{red}0}, 0], [2, 6, 3, 1, 4], [3, 6, 7, 0, 0]$ \\
& & 3 & $[0, 99, -1], [1, 1, 1]$ \\
\hline
$\multirow{4}{*}{{\rm CFI}({\rm reg}(10,3))}$ & \multirow{4}{*}{100} 
& {\color{red}2} &  $[0, 78, 7, 0, 0], [1, 12, 1, {\color{red}1}, 0], [2, 2, 7, 0, 0], [3, 6, 3, 0, 0], [4, 2, 7, 0, 0]$  \\ 
& & 3 & $[0, 99, -1], [1, 1, 1]$ \\
\cline{3-4}
& & {\color{red}2} & $[0, 78, 7, 0, 0], [1, 12, 1, {\color{red}0}, 0], [2, 2, 7, 0, 0], [3, 6, 3, 0, 0], [4, 2, 7, 0, 0]$ \\
& & 3 & $[0, 99, -1], [1, 1, 1]$ \\
\hline
$\multirow{4}{*}{{\rm CFI}({\rm reg}(10,3))}$ & \multirow{4}{*}{100} 
& {\color{red}2} &  $[0, 78, 7, 0, 0], [1, 12, 1, {\color{red}1}, 0], [2, 2, 7, 0, 0], [3, 6, 3, 0, 0], [4, 2, 1, 1, 6]$  \\ 
& & 3 & $[0, 99, 1], [1, 1, 1]$ \\
\cline{3-4}
& & {\color{red}2} & $[0, 78, 7, 0, 0], [1, 12, 1, {\color{red}0}, 0], [2, 2, 7, 0, 0], [3, 6, 3, 0, 0], [4, 2, 1, 1, 6]$ \\
& & 3 & $[0, 99, 1], [1, 1, 1]$ \\
\hline
$\multirow{4}{*}{{\rm CFI}({\rm reg}(10,3))}$ & \multirow{4}{*}{100} 
& {\color{red}2} &  $[0, 78, 7, 0, 0], [1, 12, 5, {\color{red}1}, 0], [2, 2, 7, 0, 0], [3, 8, 1, 0, 0]$  \\ 
& & 3 & $[0, 99, -1], [1, 1, 1]$ \\
\cline{3-4}
& & {\color{red}2} & $[0, 78, 7, 0, 0], [1, 12, 5, {\color{red}0}, 0], [2, 2, 7, 0, 0], [3, 8, 1, 0, 0]$ \\
& & 3 & $[0, 99, -1], [1, 1, 1]$ \\
\hline
$\multirow{4}{*}{{\rm CFI}({\rm reg}(10,3))}$ & \multirow{4}{*}{100} 
& {\color{red}2} &  $[0, 78, 7, 0, 0], [1, 12, 5, {\color{red}1}, 0], [2, 2, 7, 0, 0], [3, 8, 5, 0, 0]$  \\ 
& & 3 & $[0, 98, -1], [1, 2, -1]$ \\
\cline{3-4}
& & {\color{red}2} & $[0, 78, 7, 0, 0], [1, 12, 5, {\color{red}0}, 0], [2, 2, 7, 0, 0], [3, 8, 5, 0, 0]$ \\
& & 3 & $[0, 98, -1], [1, 2, -1]$ \\
\hline
$\multirow{4}{*}{{\rm CFI}({\rm reg}(10,3))}$ & \multirow{4}{*}{100} 
& {\color{red}2} &  $[0, {\color{red}78}, {\color{red}7}, 0, 0], [1, {\color{red}12}, 1, 0, 0], {\color{red}[2, 2, 7, 0, 0]}, [3, 8, 1, 0, 0]$  \\ 
& & 3 & $[0, 98, -1], [1, 2, -1]$ \\
\cline{3-4}
& & {\color{red}2} & $[0, {\color{red}76}, {\color{red}1}, 0, 0], [1, {\color{red}16}, 1, 0, 0], [3, 8, 1, 0, 0]$ \\
& & 3 & $[0, 98, -1], [1, 2, -1]$ \\
\hline
$\multirow{4}{*}{{\rm CFI}({\rm reg}(10,3))}$ & \multirow{4}{*}{100} 
& {\color{red}2} &  $[0, 78, 7, 0, 0], [1, 8, 1, {\color{red}1}, 0], [2, 10, 7, 0, 0], [4, 4, 5, 0, 0]$  \\ 
& & 3 & $[0, 99, -1], [1, 1, 1]$ \\
\cline{3-4}
& & {\color{red}2} & $[0, 78, 7, 0, 0], [1, 8, 1, {\color{red}0}, 0], [2, 10, 7, 0, 0], [4, 4, 5, 0, 0]$ \\
& & 3 & $[0, 99, -1], [1, 1, 1]$ \\
\hline
\end{tabular}
\caption{The $p$-adic symbols for quadratic form $A$ with respect to ${\rm CFI}({\rm reg}(10,3))$ by using Sage.}
\label{tab:p-adic_CFI2}
\end{table}
\clearpage

\begin{table}
\centering
\begin{tabular}{c||c|l|c}
\hline
\rowcolor{blue!10}
{\small Graph Type}  & $p$ & $p$-Adic Symbols & \#Graphs \\
\hline\hline
\multirow{2}{*}{srg(16, 6, 2, 2)}  & 2 & $[0, 6, 7, 0, 0], [3, 5, 7, 1, 7], [5, 4, 1, 0, 0]$ & \multirow{1}{*}{1}  \\  
\cline{2-4}
& 2 & $[0, 6, 7, 0, 0], [1, 1, 1, 1, 1], [3, 2, 7, 0, 0], [4, 2, 1, 1, 6], [5, 4, 1, 0, 0]$ & \multirow{1}{*}{1} \\
\cline{2-4}
\hline
\multirow{9}{*}{srg(25, 12, 5, 6)} & 2 & $[0, 12, 1, 0, 0], [1, 12, 1, 0, 0]$ & \multirow{3}{*}{10} \\
& 3 & $[0, 12, 1], [1, 12, 1]$ \\
& 5 & $[0, 12, -1], [1, 2, 1], [2, 10, -1]$ \\
\cline{2-4}
& 2 & $[0, 12, 5, 0, 0], [1, 12, 5, 0, 0]$ & \multirow{3}{*}{4} \\
& 3 & $[0, 12, -1], [1, 12, -1]$ \\
& 5 & $[0, 11, 1], [1, 4, 1], [2, 9, 1]$ \\
\cline{2-4}
& 2 & $[0, 12, 5, 0, 0], [1, 12, 5, 0, 0]$ & \multirow{3}{*}{1} \\
& 3 & $[0, 12, -1], [1, 12, -1]$ \\
& 5 & $[0, 9, 1], [1, 8, 1], [2, 7, 1]$ \\
\hline
\multirow{8}{*}{srg(26, 10, 3, 4)} 
& 2 & $[0, 12, 5, 0, 0], [2, 1, 1, 1, 1], [3, 12, 1, 0, 0]$ & \multirow{2}{*}{5} \\ 
& 13 & $[0, 14, -1], [1, 11, -1]$ \\
\cline{2-4}
& 2 & $[0, 12, 1, 0, 0], [2, 1, 1, 1, 1], [3, 12, 5, 0, 0]$ & \multirow{2}{*}{2} \\
& 13 & $[0, 14, 1], [1, 11, 1]$ \\
\cline{2-4}
& 2 & $[0, 12, 5, 0, 0], [2, 1, 5, 1, 5], [3, 12, 5, 0, 0]$ & \multirow{2}{*}{2} \\
& 13 & $[0, 14, -1], [1, 11, -1]$ \\
\cline{2-4}
& 2 & $[0, 12, 1, 0, 0], [2, 1, 5, 1, 5], [3, 12, 1, 0, 0]$ & \multirow{2}{*}{1} \\
& 13 & $[0, 14, 1], [1, 11, 1]$ \\
\hline
\multirow{4}{*}{srg(28, 12, 6, 4)} 
& 2 & $[0, 8, 5, 0, 0], [1, 12, 5, 0, 0], [3, 1, 5, 1, 5], [4, 6, 3, 0, 0]$ & \multirow{2}{*}{3} \\ 
& 7 & $[0, 8, 1], [1, 19, 1]$ \\
\cline{2-4}
& 2 & $[0, 6, 3, 0, 0], [1, 15, 3, 1, 1], [4, 6, 7, 0, 0]$ & \multirow{2}{*}{1} \\
& 7 & $[0, 8, 1], [1, 19, 1]$ \\
\hline
\multirow{3}{*}{srg(29, 14, 6, 7)} 
& 2 & $[0, 28, 5, 0, 0]$ & \multirow{3}{*}{41} \\ 
& 7 & $[0, 14, -1], [1, 14, -1]$ \\
& 29 & $[0, 15, 1], [1, 13, 1]$ \\
\hline
\multirow{20}{*}{srg(36, 14, 4, 6)} 
& 2 & $[0, 14, 3, 0, 0], [2, 9, 1, 1, 5], [3, 12, 1, 0, 0]$ & \multirow{2}{*}{73}  \\ 
& 3 & $[0, 14, -1], [1, 7, -1], [2, 2, 1], [3, 12, 1]$ \\
\cline{2-4}
& 2 & $[0, 14, 7, 0, 0], [2, 9, 5, 1, 5], [3, 12, 1, 0, 0]$ & \multirow{2}{*}{39}  \\ 
& 3 & $[0, 14, 1], [1, 7, 1], [2, 2, -1], [3, 12, -1]$ \\
\cline{2-4}
& 2 & $[0, 14, 3, 0, 0], [2, 9, 5, 1, 1], [3, 12, 5, 0, 0]$ & \multirow{2}{*}{36}  \\
& 3 & $[0, 14, -1], [1, 7, -1], [2, 2, 1], [3, 12, 1]$ \\
\cline{2-4}
& 2 & $[0, 12, 5, 0, 0], [1, 2, 7, 0, 0], [2, 11, 3, 1, 1], [3, 10, 3, 0, 0]$ &  \multirow{2}{*}{12} \\
& 3 & $[0, 14, -1], [1, 7, -1], [2, 2, 1], [3, 12, 1]$ \\
\cline{2-4}
& 2 & $[0, 12, 5, 0, 0], [1, 2, 7, 0, 0], [2, 11, 7, 1, 5], [3, 10, 7, 0, 0]$ & \multirow{2}{*}{8}  \\
& 3 & $[0, 14, -1], [1, 7, -1], [2, 2, 1], [3, 12, 1]$ \\
\cline{2-4}
& 2 & $[0, 14, 7, 0, 0], [2, 9, 1, 1, 1], [3, 12, 5, 0, 0]$ & \multirow{2}{*}{5}  \\
& 3 & $[0, 14, 1], [1, 7, 1], [2, 2, -1], [3, 12, -1]$ \\
\cline{2-4}
& 2 & $[0, 14, 7, 0, 0], [2, 9, 5, 1, 5], [3, 12, 1, 0, 0]$ & \multirow{2}{*}{2}  \\
& 3 & $[0, 13, -1], [1, 9, -1], [2, 1, -1], [3, 12, -1]$ \\
\cline{2-4}
& 2 & $[0, 14, 7, 0, 0], [2, 9, 5, 1, 5], [3, 12, 1, 0, 0]$ & \multirow{2}{*}{2}  \\
& 3 & $[0, 12, -1], [1, 9, -1], [2, 4, 1], [3, 10, 1]$ \\
\cline{2-4}
& 2 & $[0, 10, 3, 0, 0], [1, 4, 1, 0, 0], [2, 13, 5, 1, 1], [3, 8, 5, 0, 0]$ & \multirow{2}{*}{2}  \\
& 3 & $[0, 14, -1], [1, 7, -1], [2, 2, 1], [3, 12, 1]$ \\
\cline{2-4}
& 2 & $[0, 8, 5, 0, 0], [1, 6, 7, 0, 0], [2, 15, 3, 1, 1], [3, 6, 3, 0, 0]$ & \multirow{2}{*}{1}  \\
& 3 & $[0, 14, -1], [1, 7, -1], [2, 2, 1], [3, 12, 1]$ \\
\hline
\end{tabular}
\caption{The $p$-adic symbols for quadratic form of combinatorial Laplacian with respect to ${\rm srg}(n,k,\lambda,\mu)$ by using Sage.}
\label{tab:p-adic_srg_lap}
\end{table}
\clearpage

\begin{table}
\centering
\begin{tabular}{c||c|l|c}
\hline
\rowcolor{blue!10}
{\small Graph Type}  & $p$ & $p$-Adic Symbols & \#Graphs \\
\hline\hline
\multirow{12}{*}{srg(40, 12, 2, 4)} 
& 2 & $[0, 16, 5, 0, 0], [1, 8, 1, 0, 0], [3, 1, 3, 1, 3], [5, 14, 3, 0, 0]$ & \multirow{2}{*}{17} \\
& 5 & $[0, 16, -1], [1, 23, 1]$ \\ 
\cline{2-4}
& 2 & $[0, 14, 7, 0, 0], [1, 10, 3, 0, 0], [3, 1, 7, 1, 7], [4, 2, 7, 0, 0], [5, 12, 1, 0, 0]$ & \multirow{2}{*}{5} \\
& 5 & $[0, 16, -1], [1, 23, 1]$ \\ 
\cline{2-4}
& 2 & $[0, 14, 7, 0, 0], [1, 10, 3, 0, 0], [3, 1, 3, 1, 3], [4, 2, 3, 0, 0], [5, 12, 1, 0, 0]$ & \multirow{2}{*}{3} \\
& 5 & $[0, 16, -1], [1, 23, 1]$ \\ 
\cline{2-4}
& 2 & $[0, 12, 1, 0, 0], [1, 12, 5, 0, 0], [3, 1, 3, 1, 3], [4, 4, 5, 0, 0], [5, 10, 7, 0, 0]$ & \multirow{2}{*}{1} \\
& 5 & $[0, 16, -1], [1, 23, 1]$ \\ 
\cline{2-4}
& 2 & $[0, 12, 1, 0, 0], [1, 12, 5, 0, 0], [3, 1, 7, 1, 7], [4, 4, 1, 0, 0], [5, 10, 7, 0, 0]$ & \multirow{2}{*}{1} \\
& 5 & $[0, 16, -1], [1, 23, 1]$ \\ 
\cline{2-4}
& 2 & $[0, 10, 7, 0, 0], [1, 14, 3, 0, 0], [3, 1, 7, 1, 7], [4, 6, 7, 0, 0], [5, 8, 1, 0, 0]$ & \multirow{2}{*}{1} \\
& 5 & $[0, 16, -1], [1, 23, 1]$ \\ 
\hline
\multirow{27}{*}{srg(45, 12, 3, 3)} 
& 2 & $[0, 44, 5, 0, 0]$ & \multirow{3}{*}{19} \\
& 3 & $[0, 20, 1], [1, 4, -1], [2, 2, -1], [3, 18, -1]$ \\
& 5 & $[0, 21, -1], [1, 23, -1]$ \\
\cline{2-4}
& 2 & $[0, 44, 5, 0, 0]$ & \multirow{3}{*}{18} \\
& 3 & $[0, 19, -1], [1, 6, 1], [2, 1, -1], [3, 18, -1]$ \\
& 5 & $[0, 21, -1], [1, 23, -1]$ \\
\cline{2-4}
& 2 & $[0, 44, 5, 0, 0]$ & \multirow{3}{*}{13} \\
& 3 & $[0, 20, -1], [1, 4, 1], [2, 2, 1], [3, 18, 1]$ \\
& 5 & $[0, 21, -1], [1, 23, -1]$ \\
\cline{2-4}
& 2 & $[0, 44, 5, 0, 0]$ & \multirow{3}{*}{8} \\
& 3 & $[0, 19, 1], [1, 6, -1], [2, 1, 1], [3, 18, 1]$ \\
& 5 & $[0, 21, -1], [1, 23, -1]$ \\
\cline{2-4}
& 2 & $[0, 44, 5, 0, 0]$ & \multirow{3}{*}{6} \\
& 3 & $[0, 18, 1], [1, 6, -1], [2, 4, -1], [3, 16, -1]$ \\
& 5 & $[0, 21, -1], [1, 23, -1]$ \\
\cline{2-4}
& 2 & $[0, 44, 5, 0, 0]$ & \multirow{3}{*}{6} \\
& 3 & $[0, 17, 1], [1, 8, -1], [2, 3, 1], [3, 16, 1]$ \\
& 5 & $[0, 21, -1], [1, 23, -1]$ \\
\cline{2-4}
& 2 & $[0, 44, 5, 0, 0]$ & \multirow{3}{*}{3} \\
& 3 & $[0, 18, -1], [1, 6, 1], [2, 4, 1], [3, 16, 1]$ \\
& 5 & $[0, 21, -1], [1, 23, -1]$ \\
\cline{2-4}
& 2 & $[0, 44, 5, 0, 0]$ & \multirow{3}{*}{3} \\
& 3 & $[0, 17, -1], [1, 8, 1], [2, 3, -1], [3, 16, -1]$ \\
& 5 & $[0, 21, -1], [1, 23, -1]$ \\
\cline{2-4}
& 2 & $[0, 44, 5, 0, 0]$ & \multirow{3}{*}{2} \\
& 3 & $[0, 15, -1], [1, 10, 1], [2, 5, -1], [3, 14, -1]$ \\
& 5 & $[0, 21, -1], [1, 23, -1]$ \\
\hline
\end{tabular}
\caption{The $p$-adic symbols for quadratic form of combinatorial Laplacian with respect to ${\rm srg}(n,k,\lambda,\mu)$ by using Sage.}
\label{tab:p-adic_srg_lap_2}
\end{table}
\clearpage


\section{Codes}
\subsection{Magma Code for Checking Isomorphism of Quadratic Forms} \label{code:magma}
$~$
\begin{lstlisting}
function CheckQuadIso(A1, A2, dim)
    A1 := MatrixRing(IntegerRing(), dim) ! A1;
    A2 := MatrixRing(IntegerRing(), dim) ! A2;
    L1 := LatticeWithGram(A1);
    L2 := LatticeWithGram(A2);
    rslt := IsIsometric(L1, L2);
    return rslt;
end function;
\end{lstlisting}

\subsection{Sage Code for Computing $p$-Adic Symbols of Quadratic Forms} \label{code:sage_padic}
$~$
\begin{lstlisting}
def pAdic(A, dim):
    A = matrix(ZZ, dim, A);
    Q = QuadraticForm(ZZ, A);
    rslt = Q.CS_genus_symbol_list();
    return rslt;
\end{lstlisting}

\subsection{Sage Code for Checking Local Equivalence of Quadratic Forms} \label{code:sage_local}
$~$
\begin{lstlisting}
def CheckLocalIso(A1, A2, dim):
    A1 = matrix(ZZ, dim, A1);
    A2 = matrix(ZZ, dim, A2);
    Q1 = QuadraticForm(ZZ, A1);
    Q2 = QuadraticForm(ZZ, A2);
    rslt = Q1.is_locally_equivalent_to(Q2, check_primes_only=True);
    return rslt;
\end{lstlisting}

\end{document}